\documentclass[final]{siamonline250211}

\usepackage{amsfonts,amssymb,amsmath}
\usepackage{booktabs}
\usepackage{enumitem}
\usepackage{graphicx}
\usepackage{url}
\usepackage{comment}

\newcommand{\R}{\mathbb{R}}
\newcommand{\one}{\mathbf 1}
\newcommand{\KL}{\operatorname{KL}}
\newcommand{\diag}{\operatorname{diag}}
\newcommand{\lse}{\operatorname{LSE}}
\newcommand{\Smap}{\mathcal S}
\newcommand{\lm}{\operatorname{lm}}
\newsiamremark{remark}{Remark}

\usepackage{xcolor}

\headers{ACC-SINKHORN}{L. CHEN AND Z. XU}

\title{Accelerating Sinkhorn for Entropy-Regularized Optimal Transport}

\author{Long Chen\thanks{Department of Mathematics, University of California, Irvine, CA 92697, USA.}
\and Zeyi Xu\thanks{Department of Mathematics, University of California, Irvine, CA 92697, USA.}}

\begin{document}

\maketitle

\begin{abstract}
Entropy-regularized optimal transport is widely used to compare and align data distributions, but Sinkhorn can converge slowly at small regularization. We introduce Acc-Sinkhorn, a two-state Hessian-driven Nesterov iteration built on the dual mirror geometry of the row-reduced Sinkhorn map. The method combines momentum acceleration with an inexpensive descent safeguard, stable log-domain evaluation, and a prescribed homotopy schedule requiring neither line search nor curvature estimation. We prove global linear convergence of the safeguarded method, rule out consecutive restarts, and establish local accelerated convergence for fixed parameters without restarts through spectral analysis. 
Numerical comparisons with six methods on synthetic problems, color transfer, and bilingual word alignment show runtime speedups of up to about $26$ over Sinkhorn, with additional rendering-accuracy comparisons for color transfer.
\end{abstract}



\tableofcontents

\section{Introduction}
\label{sec:introduction}

Optimal transport provides a way to compare and align data distributions using a cost that reflects their underlying geometry. Entropy regularization makes this computation amenable to the Sinkhorn algorithm, enabling its use in generative modeling, domain adaptation, image processing, computer vision, and natural language processing~\cite{cuturi2013sinkhorn,genevay2018learning,courty2017joint,solomon2015convolutional,xu2019learning,alvarez2020unsupervised}. The efficiency and reliability of the transport solve therefore matter to the computational methods built on it. This paper develops an accelerated solver that retains Sinkhorn's basic computational structure.

Let $a\in\mathbb{R}_{++}^n$ and $b\in\mathbb{R}_{++}^m$ be probability vectors. The \emph{entropy-regularized optimal transport} (EOT) problem is
\begin{equation}
\min_{P\in\Pi(a,b)}
\langle C,P\rangle+
\varepsilon\sum_{i,j}P_{ij}(\log P_{ij}-1),
\label{eq:EOT}
\end{equation}
where $C\in\mathbb{R}_+^{n\times m}$ is a cost matrix, $\varepsilon>0$ is the regularization parameter, and
\begin{equation*}
\Pi(a,b):=\left\{P\in\mathbb{R}_+^{n\times m}:\;P\mathbf{1}_m=a,\  P^\top\mathbf{1}_n=b\right\}
\end{equation*}
is the transportation polytope.
The unregularized problem is
\begin{equation}
\min_{P\in\Pi(a,b)}\langle C,P\rangle.
\label{eq:OT}
\end{equation}
Its transport plan has $nm$ entries, making generic linear programming approaches costly for large problems~\cite{peyre2019computational}. For $\varepsilon>0$, EOT has a unique positive solution. Sinkhorn~\cite{sinkhorn1964} computes it by alternately rescaling the rows and columns of the Gibbs kernel $K=\exp(-C/\varepsilon)$. Each iteration costs $\mathcal O(nm)$ operations, admits parallel evaluation, and is stable in the log domain~\cite{schmitzer2019stabilized}. These properties have made Sinkhorn a standard component of EOT software~\cite{reinger2021pot}.

Small regularization can produce sharper couplings, but it also concentrates the Gibbs kernel and can substantially increase the number of Sinkhorn steps. 
Our objective is to reduce the number of iterations while preserving stable evaluation and providing a fallback when acceleration fails. Smaller regularization need not improve every downstream task; we therefore distinguish transport feasibility from application accuracy in the experiments.

Several approaches address this computational difficulty. APDAGD and APDAMD apply accelerated first-order methods to transport duals with adaptive backtracking~\cite{dvurechensky2018computational,lin2019efficient}. Overrelaxed Sinkhorn extrapolates the scaling steps~\cite{thibault2021overrelaxed}, and annealed Sinkhorn varies the regularization~\cite{chizat2024annealed}. Scaling0 of Allen-Zhu et al.~\cite{allenzhu2017scaling} accelerates fixed-matrix scaling under a weighted marginal-residual criterion. More closely related to our method, accelerated alternating minimization~\cite{guminov2021combination} and accelerated Sinkhorn variants~\cite{lin2019acceleration} combine exact block minimization with Nesterov's estimate sequences. Gao et al.~\cite{gao2026tight} apply Nesterov acceleration to a diagonally preconditioned semi-dual; their momentum parameter depends on the local curvature constant $\mu_\star$, and their analysis starts from a warm start.

Our starting point is the mirror geometry of Sinkhorn. Its convergence and mirror-descent interpretations have been studied extensively~\cite{franklin1989scaling,carlier2022linear,ghosal2025convergence,chizat2026sharper,pmlr-v238-reza-karimi24a,mishchenko2019sinkhorn,leger2021gradient,aubin2022mirror,pmlr-v235-chopin24a,maddison2021dual}. Eliminating the row potential $u$ from the dual objective $F(u,v)$ gives $f(v)=\min_u F(u,v)$. If $c(v)$ denotes the column marginal after this row solve, then $\nabla f(v)=c(v)-b$, and Sinkhorn is the unit-step dual mirror update
\begin{equation}
v_{k+1}=v_k-\nabla\phi^*(\nabla f(v_k))
=v_k-\log\bigl(c(v_k)./b\bigr).
\label{eq:intro-sinkhorn-mirror}
\end{equation}
The logarithm and division act componentwise. This interpretation leads to acceleration.

Our contributions are as follows.
\begin{enumerate}[leftmargin=2em]
\item We propose Acc-Sinkhorn, a two-state method that combines the Sinkhorn map with Hessian-driven Nesterov acceleration (HNAG)~\cite{chen2019orderoptimizationmethodsbased}:
\begin{equation}
\begin{aligned}
\frac{x_{k+1}-x_k}{\alpha}
&=
y_k-x_{k+1}-\beta \pi_{\mathsf V}\nabla\phi^*(\nabla f(x_k)),\\
\frac{y_{k+1}-y_k}{\alpha}
&=
x_{k+1}-y_{k+1}-\frac1\gamma \pi_{\mathsf V}\nabla\phi^*(\nabla f(x_{k+1})),
\end{aligned}
\label{intro:acc-sinkhorn-scheme}
\end{equation}
where $\pi_{\mathsf V}$ is the orthogonal projection onto $\mathsf V=\{v\in\R^m:\one_m^\top v=0\}$, the subspace of mean-zero vectors.
We set $\alpha\beta =1$ and $\alpha = \sqrt{2\gamma}$. An iteration of \eqref{intro:acc-sinkhorn-scheme} is accepted only if it decreases $f$ sufficiently; otherwise the method takes one Sinkhorn step and resets $y_{k+1}=x_{k+1}$. The test adds no matrix pass when it succeeds. A restart costs one additional row reduction. The homotopy decreases $\gamma$ and carries the auxiliary state between stages, without line search or problem-dependent tuning. 

\item We prove global linear convergence of the safeguarded method from any initial potential and show that restarts never occur twice in a row. For fixed $0<\gamma\leq\mu_\star$ without restarts, spectral analysis bounds the linearized momentum iteration's spectral radius by $(1+\sqrt{2\gamma})^{-1}$. A separate energy argument gives a global $\mathcal O(k^{-2})$ estimate for the homotopy when restarts per stage are uniformly bounded.

\item We compare Acc-Sinkhorn with six published methods under a common work unit. The tests include synthetic problems, color transfer, and bilingual word alignment. Acc-Sinkhorn has the shortest runtime at the smaller regularization values tested. At a common marginal tolerance, the reported runtime speedup over Sinkhorn reaches about $26$. The color-transfer study also measures runtime against rendering accuracy. 
\end{enumerate}

The paper is organized as follows. We first derive the row-reduced EOT dual, identify Sinkhorn as dual mirror descent, and verify dual relative smoothness. We then motivate Acc-Sinkhorn from HNAG and present its fixed-parameter and homotopy forms. Section~\ref{sec:restart} states the convergence results. Appendix~\ref{app:local-convergence} gives the detailed homotopy analysis. Section~\ref{sec:experiments} compares the methods numerically. 

\section{Sinkhorn and dual mirror geometry}
\label{sec:reduced-formulation}
We derive the reduced dual and identify Sinkhorn as dual mirror descent. All operations $.*$, $./$, $\exp$, and $\log$ act componentwise.

\subsection{Reduced dual and Sinkhorn map}

Let $a\in\R^n_{++}$ and $b\in\R^m_{++}$ be probability vectors, let $C\in\R^{n\times m}$ be finite, and let $\varepsilon>0$. Let $\bar C=C/\varepsilon$. As dividing the objective by $\varepsilon$ does not change its minimizer, we consider
\begin{equation}
\min_{P\geq0}\big{\{}
\langle \bar C,P\rangle+
\sum_{i,j}P_{ij}(\log P_{ij}-1):
P\one_m=a,\;P^\top\one_n=b
\big{ \}}.
\label{eq:eot}
\end{equation}
Introducing Lagrange multipliers $u\in\R^n$ and $v\in\R^m$, define
\begin{equation*}
\mathcal{L}(P,u,v)
=
\langle \bar C,P\rangle
+\sum_{i,j}P_{ij}(\log P_{ij}-1)
+\langle u,a-P\one_m\rangle
+\langle v,b-P^\top\one_n\rangle.
\end{equation*}
Solving the stationarity condition $\partial_P\mathcal{L}(P,u,v)=0$ gives
\begin{equation}
P_{ij}(u,v)=\exp(u_i+v_j-\bar C_{ij}).
\label{eq:P-from-dual}
\end{equation}
Substituting \eqref{eq:P-from-dual} into the Lagrangian gives the convex dual objective
\begin{equation}
\begin{aligned}
F(u,v)
&:=-\mathcal L(P(u,v),u,v)=\sum_{i,j}P_{ij}(u,v)
-\langle a,u\rangle-\langle b,v\rangle.
\end{aligned}
\label{eq:dual-obj}
\end{equation}
Strong duality allows us to recover the solution of \eqref{eq:eot} by solving the dual problem 
\begin{equation}
\min_{u\in\R^n,\,v\in\R^m}F(u,v).
\label{eq:dual-min}
\end{equation}

Let
$r_P(u,v):=P(u,v)\one_m,$ $
c_P(u,v):=P(u,v)^\top\one_n$ be the row and column sums of $P$. 
Differentiating \eqref{eq:dual-obj} and using $\partial_P\mathcal{L}(P,u,v)=0$  gives
\begin{equation}
\partial_u F(u,v)=r_P(u,v)-a,
\qquad
\partial_v F(u,v)=c_P(u,v)-b.
\label{eq:full-gradient}
\end{equation}
Hence $\nabla F(u,v)=0$ is equivalent to the two marginal constraints.

For fixed $v$, $F(\cdot,v)$ separates over the coordinates of $u$ and is strictly convex. Its minimizer $u(v)$ satisfies $r_P(u(v),v)=a$.
Solving this condition gives
\begin{equation}
u_i(v)=\log a_i-\log\left(\sum_{j=1}^m
\exp(v_j-\bar C_{ij})\right).
\label{eq:row-solve}
\end{equation}
This update enforces the row marginal exactly. Its iterative form is 
\begin{equation}
u^{k+1}=u^k-\log\left(r_P(u^k,v^k)\,./\,a\right).
\label{eq:u-update}
\end{equation}
Applying the same argument to the $v$-block gives
\begin{equation}
v^{k+1}=v^k-\log\left(c_P(u^{k+1},v^k)\,./\,b\right).
\label{eq:v-update}
\end{equation}
After the two updates, $r_P(u^{k+1},v^k)=a,$ $ c_P(u^{k+1},v^{k+1})=b.$
Equations \eqref{eq:u-update}--\eqref{eq:v-update} are the Sinkhorn algorithm: alternating exact minimization of $F(u,v)$ over $u$ and $v$. This interpretation is classical~\cite{benamou2015iterative}.

Since $\one_n^\top a=\one_m^\top b=1$, for every $t\in\R$,
\begin{equation}
\begin{aligned}
F(u+t\one_n,v-t\one_m)
&=
F(u,v)-t\,\one_n^\top a+t\,\one_m^\top b
=
F(u,v).
\end{aligned}
\label{eq:Fc}
\end{equation}
To see that this is the only ambiguity, compute 
\begin{equation}
\nabla^2F(u,v)
=
\sum_{i=1}^n\sum_{j=1}^m\nabla^2P_{ij}
=
\begin{pmatrix}
\operatorname{diag}(r_P) & P\\
P^\top & \operatorname{diag}(c_P)
\end{pmatrix}.
\label{eq:dual-hessian}
\end{equation}
Since $P(u,v)$ has positive entries, the Hessian of $F$ has kernel
$\operatorname{span}\left\{
(\one_n, -\one_m)^{\top}
\right\}.$
Thus the minimizer of \eqref{eq:dual-min} is unique modulo this shift. We fix, for example, the gauge $\langle v,\one_m\rangle=0$.

\subsection{Bilevel formulation}
By transposing the problem if necessary, assume $m\leq n$ so that eliminating $u$ leaves the smaller variable $v$. We assume the nontrivial case $m\geq2$. Define $P_v:=P(u(v),v).$
Since $u(v)$ solves the row subproblem exactly,
$P_v\one_m=a.$
Because $a$ is a probability vector,
$\sum_{i,j}(P_v)_{ij}
=
\one_n^\top P_v\one_m
=
\one_n^\top a
=
1.
$
Therefore, we introduce the reduced dual function
\begin{equation}
f(v)
:=
F(u(v),v)
=
1-\langle a,u(v)\rangle-\langle b,v\rangle.
\label{eq:reduced-dual}
\end{equation}

Since $\partial_u F(u(v),v)=0$, the chain rule gives
\begin{equation}
\begin{aligned}
\nabla f(v)
&=
\left(u'(v)\right)^\top
\partial_u F(u(v),v)
+\partial_v F(u(v),v) =
\partial_v F(u(v),v)
=
c(v)-b,
\end{aligned}
\label{eq:grad-f}
\end{equation}
where $c(v):=P(u(v),v)^\top\one_n.$
Thus the gradient of the reduced dual is the remaining column-marginal residual.

We next compute the Hessian of $f$. 
Differentiating the optimality condition $\partial_u F(u(v),v)=0$ with respect to $v$ gives
$
u'(v)
=
-\bigl(\nabla_{uu}^2F\bigr)^{-1}\nabla_{uv}^2F.$
A second differentiation of $\nabla f(v)=\partial_v F(u(v),v)$ yields
\begin{equation*}
\begin{aligned}
\nabla^2f(v)
&=
\partial_{vv}^2F+\partial_{vu}^2F\,u'(v)=
\partial_{vv}^2F
-\partial_{vu}^2F
\bigl(\nabla_{uu}^2F\bigr)^{-1}
\nabla_{uv}^2F.
\end{aligned}
\end{equation*}
All derivatives are evaluated at $(u(v),v)$. Using \eqref{eq:dual-hessian}
and $r_P(u(v),v)=a$, the reduced Hessian $\nabla^2f(v)$ is the Schur complement of the $u$-block:
\begin{equation}
\nabla^2f(v)
=
\operatorname{diag}(c(v))
-P^\top\operatorname{diag}(a)^{-1}P
\preceq
\operatorname{diag}(c(v)).
\label{eq:HessianL}
\end{equation}
The inequality holds because $P^\top\operatorname{diag}(a)^{-1}P$ is positive semidefinite.

\subsection{Mirror function}

The Hessian bound \eqref{eq:HessianL} suggests a dual mirror function with inverse diagonal Hessian. Writing $\chi=\nabla f(v)=c(v)-b$ gives
\begin{equation}
\nabla^2\phi^*(\chi)
=\operatorname{diag}(b+\chi)^{-1}
=\operatorname{diag}(c(v))^{-1},
\qquad \chi>-b.
\label{eq:dual-mirror-hessian}
\end{equation}
We recover the mirror map by integrating \eqref{eq:dual-mirror-hessian} coordinatewise with the normalization $\nabla\phi^*(0)=0$:
\begin{equation}
\nabla\phi^*(\chi)
=
\log(\one_m+\chi./b),
\qquad
\chi>-b.
\label{eq:dual-mirror-map}
\end{equation}
Integrating once more and choosing $\phi^*(0)=0$ gives
\begin{equation*}
\phi^*(\chi)=\sum_{j=1}^m\left[(b_j+\chi_j)\log\frac{b_j+\chi_j}{b_j}-\chi_j\right],
\qquad \chi>-b.
\end{equation*}
This is negative entropy shifted by $b$, up to an affine term. Since $c(v)>0$, the domain condition holds at $\chi=\nabla f(v)=c(v)-b$, and
\begin{equation*}
\nabla\phi^*(\nabla f(v))
=
\log(c(v)./b).
\end{equation*}
Thus the unit-step dual mirror-descent update
\begin{equation}
\begin{aligned}
v^{k+1}
&=
v^k-\nabla\phi^*(\nabla f(v^k)) =
v^k-\log\bigl(c_P(u^{k+1},v^k)./b\bigr)
\end{aligned}
\label{eq:md-general}
\end{equation}
is the Sinkhorn $v$-update \eqref{eq:v-update}. Figure~\ref{fig:sinkhorn-direction-geometry} illustrates this relation coordinatewise.

\begin{figure}[htbp]
\begin{center}
\includegraphics[width=2.5in]{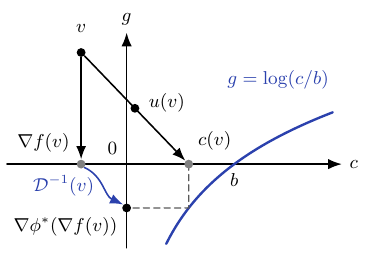}
\caption{Coordinatewise Sinkhorn geometry: the row solve $u(v)$ gives the column marginal $c(v)$, gradient $\nabla f(v)=c(v)-b$, and direction $\nabla\phi^*(\chi)=\log(c(v)/b)$.}
\label{default}
\label{fig:sinkhorn-direction-geometry}
\end{center}
\end{figure}

\begin{remark}
Our formulation expresses L\'eger's marginal-space mirror descent~\cite{leger2021gradient} in the dual potential $v$, yielding the same Sinkhorn iteration \eqref{eq:md-general}. We use $v$ because it is unconstrained, whereas the marginal lies in the probability simplex; momentum and extrapolation therefore apply directly to $v$.
\end{remark}

\subsection{Normalized Sinkhorn map}

The reduced dual $f$ is invariant under constant shifts:
\begin{equation*}
f(v+t\one_m)=f(v),
\qquad
t\in\R.
\end{equation*}
We therefore work on the quotient by $\operatorname{span}\{\one_m\}$ and use the Euclidean gauge
\begin{equation}
\mathsf V =\{v\in\R^m:\one_m^\top v=0\},
\qquad
\pi_{\mathsf V}v=v-\frac{\one_m^\top v}{m}\one_m.
\label{eq:gauge}
\end{equation}
All inverse Hessians and relative-smoothness inequalities below are understood on $\mathsf V$.

The normalized Sinkhorn map is
\begin{equation}
\begin{aligned}
\Smap(v)
&:=
\pi_{\mathsf V}\left(v-\nabla\phi^*(\nabla f(v))\right)=
\pi_{\mathsf V}\left(v+\log b-\log c(v)\right),
\qquad
v\in \mathsf V.
\end{aligned}
\label{eq:sinkhorn-map}
\end{equation}
The projection changes the potential only by a constant and leaves the coupling unchanged.

\subsection{Stable log-domain evaluation}
\label{sec:stable-evaluation}

At small $\varepsilon$, entries of $e^{-C/\varepsilon}$ can underflow~\cite{schmitzer2019stabilized}. We evaluate the row solve and log column marginal directly, without forming the kernel or the reduced gradient. Define
\begin{equation}
\lse(z_1,\ldots,z_q)
:=
z_{\max}
+\log\left(\sum_{\ell=1}^q e^{z_\ell-z_{\max}}\right)
=
\log\left(\sum_{\ell=1}^q e^{z_\ell}\right),
\label{eq:lse}
\end{equation}
where $z_{\max}:=\max_{1\leq\ell\leq q}z_\ell$. All exponentials in the shifted sum lie in $(0,1]$. The row solve and column marginal are evaluated by
\begin{align}
u_i(v)
&=
\log a_i-\lse_{1\leq j\leq m}(v_j-\bar C_{ij}),
\label{eq:stable-u}\\
\log c_j(v)
&=
v_j+\lse_{1\leq i\leq n}(u_i(v)-\bar C_{ij}).
\label{eq:stable-c}
\end{align}
Hence the normalized Sinkhorn map $\Smap(v)$
can be evaluated without forming $e^{-\bar C}$, $P(u(v),v)$, or $\nabla f(v)=c(v)-b$. The evaluation requires one row-wise and one column-wise reduction over the $n\times m$ array and therefore costs $\mathcal O(nm)$ operations.

The reduced objective requires no additional matrix pass. Since the row solve gives $P(u(v),v)\one_m=a$, the total mass is one, and
$
f(v)
=
1-\langle a,u(v)\rangle-\langle b,v\rangle.$
After $u(v)$ has been computed during the evaluation of $\Smap(v)$, evaluating $f(v)$ requires only $\mathcal O(n+m)$ vector operations. Using the exact total mass avoids the loss of underflowed entries that may occur when summing $P(u(v),v)$ numerically.

\subsection{Verification of dual relative smoothness}
\label{sec:dual-relative-smoothness}

The Fenchel conjugate $f^*$ is understood as the conjugate of the restriction $f|_{\mathsf V}$. Indeed, the shift invariance of $f$ makes the unreduced conjugate finite only on the subspace orthogonal to $\one_m$.

If both Hessians in \eqref{eq:HessianL} were positive definite on $\R^m$, inversion would give \eqref{eq:dual-relative-smoothness}.
This argument is not directly valid because $\one_m\in\ker\nabla^2f(v)$. We therefore verify the inequality on the quotient by constant vectors. 

\begin{proposition}[Dual relative smoothness]
\label{prop:dual-relative-smoothness}
On the quotient by constant vectors, $f$ is $1$-dual smooth relative to $\phi$. More precisely, if $\chi=\nabla f(v)$, then
\begin{equation}
\nabla^2\phi^*(\chi)|_{\mathsf V}
\preceq
\nabla^2f^*(\chi)|_{\mathsf V}.
\label{eq:dual-relative-smoothness}
\end{equation}
The inequality is understood in the sense of quadratic forms on $\mathsf V$.
\end{proposition}

\begin{proof}
Let $P=P(u(v),v)$, $c=P^\top\one_n$, $D_a:=\diag(a)$, and $D_c:=\diag(c)$. The reduced Hessian satisfies
\begin{equation}
H:=\nabla^2f(v)=D_c-P^\top D_a^{-1}P\preceq D_c.
\label{eq:reduced-Hessian}
\end{equation}
For $q_i:=P_{i:}^\top/a_i$, each $q_i$ is a probability vector with full support, and
\begin{equation}
z^\top Hz=\sum_{i=1}^n a_i\operatorname{Var}_{q_i}(z),
\qquad z\in\R^m.
\label{eq:Hessian-variance-short}
\end{equation}
Thus $\ker H=\operatorname{span}\{\one_m\}$ and $H|_{\mathsf V}\succ0$.

For every $\zeta\in\mathsf V$, the variational characterization of the inverse gives
\begin{equation*}
\begin{aligned}
\zeta^\top(H|_{\mathsf V})^{-1}\zeta
&=\sup_{z\in\mathsf V}\{2\zeta^\top z-z^\top Hz\} =\sup_{z\in\R^m}\{2\zeta^\top z-z^\top Hz\}\\
&\geq\sup_{z\in\R^m}\{2\zeta^\top z-z^\top D_cz\}
=\zeta^\top D_c^{-1}\zeta.
\end{aligned}
\end{equation*}
The second equality holds because $\zeta^\top\one_m=0$ and $H\one_m=0$, so the expression is invariant under adding a constant vector. The inequality follows from \eqref{eq:reduced-Hessian}.

Finally, at $\chi=\nabla f(v)=c-b$, Legendre duality on $\mathsf V$ and the definition of the mirror function give
\begin{equation*}
\nabla^2f^*(\chi)|_{\mathsf V}=(H|_{\mathsf V})^{-1},
\qquad
\nabla^2\phi^*(\chi)=D_c^{-1}.
\end{equation*}
The preceding quadratic-form inequality proves \eqref{eq:dual-relative-smoothness}.
\end{proof}

Integrating the Hessian inequality along the segment from $\xi$ to $\eta$ gives the following Bregman bound. Here the Bregman divergence of $h$ is $D_h(z,w):=h(z)-h(w)-\langle\nabla h(w),z-w\rangle$.
\begin{corollary}[Bregman form of dual relative smoothness]
\label{cor:dual-relative-smoothness-bregman}
For all $\xi,\eta\in\operatorname{int}(\operatorname{dom}f^*)\subset \mathsf V$,
\begin{equation}
D_{\phi^*}(\eta,\xi)
\leq
D_{f^*}(\eta,\xi).
\label{eq:dual-relative-smoothness-bregman}
\end{equation}
\end{corollary}

\subsection{Kullback--Leibler divergence}

All Bregman divergences and conjugates in this subsection are understood on the quotient by constant vectors. For $p\in\R_+^m$ and $q\in\R_{++}^m$, define the generalized relative entropy
\begin{equation*}
\KL(p\|q)
:=
\sum_{j=1}^m
\left[
p_j\log\left(\frac{p_j}{q_j}\right)-p_j+q_j
\right],
\end{equation*}
with the convention $0\log 0=0$. For probability vectors $p$ and $q$, this is the usual Kullback--Leibler divergence. The same definition applies to matrices by summing over their entries.

The affine term does not affect the Bregman divergence, so the shifted entropy gives
\begin{equation*}
D_{\phi^*}(\eta,\xi)
=
\KL(b+\eta\|b+\xi).
\end{equation*}
Since $\nabla f(v)=c(v)-b$, we obtain the following identities.
\begin{lemma}[Mirror Bregman divergences]
\label{lem:mirror-Bregman-KL}
For every $v\in\mathbb R^m$,
\begin{equation}
\begin{aligned}
D_{\phi^*}(\nabla f(v),0)
&=
\KL(c(v)\|b),\\
D_{\phi^*}(0,\nabla f(v))
&=
\KL(b\|c(v)).
\end{aligned}
\label{eq:mirror-Bregman-KL}
\end{equation}
\end{lemma}
The two quantities are generally different. In particular, $\KL(b\|c(v))$ can be evaluated from the log-sum-exp values of $\log c(v)$ without forming the Gibbs kernel.

\begin{lemma}[Column violation from the dual gap]
\label{lem:residual-from-gap}
Let $P_v:=P(u(v),v)$, and let $P_\star:=P_{v^\star}$ be the optimal plan of \eqref{eq:eot}. For every $v\in \mathsf V$,
\begin{equation}
D_f(v,v^\star)
=
f(v)-f(v^\star)
=
\KL(P_\star\|P_v).
\label{eq:gap-is-kl}
\end{equation}
Consequently,
\begin{equation}
\frac12\|c(v)-b\|_1^2
\leq
\KL(b\|c(v))
\leq
\KL(P_\star\|P_v).
\label{eq:residual-from-gap}
\end{equation}
\end{lemma}

\begin{proof}
As $\nabla f(v^\star)=0$, $D_f(v,v^\star)
= f(v)-f(v^\star)$. 
By \eqref{eq:reduced-dual},
\begin{equation*}
f(v)-f(v^\star)
=
-\langle a,u(v)-u(v^\star)\rangle
-\langle b,v-v^\star\rangle.
\end{equation*}
Since $P_\star\one_m=a$ and $P_\star^\top\one_n=b$,
\begin{equation*}
\begin{aligned}
\langle a,u(v)-u(v^\star)\rangle
+\langle b,v-v^\star\rangle &=
\sum_{i,j}(P_\star)_{ij}
\left[
u_i(v)-u_i(v^\star)+v_j-v_j^\star
\right]\\
&=
\sum_{i,j}(P_\star)_{ij}
\log\left(\frac{(P_v)_{ij}}{(P_\star)_{ij}}\right)
=
-\KL(P_\star\|P_v).
\end{aligned}
\end{equation*}
This proves \eqref{eq:gap-is-kl}.

Pinsker's inequality gives
$
\frac12\|c(v)-b\|_1^2
\leq
\KL(b\|c(v)).
$
Moreover, dual relative smoothness and Bregman duality give
\begin{equation*}
\begin{aligned}
\KL(b\|c(v))
&=
D_{\phi^*}(0,\nabla f(v)) \leq
D_{f^*}(0,\nabla f(v)) =
D_f(v,v^\star)
=
\KL(P_\star\|P_v).
\end{aligned}
\end{equation*}
This proves \eqref{eq:residual-from-gap}.
\end{proof}

On an initial sublevel set, a Polyak--\L ojasiewicz (PL) bound controls the objective gap by either marginal KL divergence. Fix
\begin{equation}\label{eq:K0}
\mathcal K_0
:=
\{v\in \mathsf V:f(v)\leq f(v_0)\}.
\end{equation}
The set $\mathcal K_0$ is compact by coercivity of $f|_{\mathsf V}$ and convex by convexity of $f$. Define
\begin{equation}
\lambda_{\min}
:=
\inf_{v\in\mathcal K_0}
\min_{0\ne h\in \mathsf V}
\frac{h^\top\nabla^2f(v)h}{\|h\|_2^2}>0,
\qquad
M_c
:=
\sup_{v\in\mathcal K_0}\|c(v)\|_\infty\leq1.
\label{eq:consts}
\end{equation}
Since $v^\star\in\mathcal K_0$ and $c(v^\star)=b$, both $b_j$ and $c_j(v)$ are bounded above by $M_c$ for $v\in\mathcal K_0$.

\begin{proposition}
\label{prop:pl}
Set $\mu:=\lambda_{\min}/M_c$. Then, for every $v\in\mathcal K_0$,
\begin{equation}
f(v)-f(v^\star)
\leq
\frac1\mu
\min\left\{
D_{\phi^*}(0,\nabla f(v)),
D_{\phi^*}(\nabla f(v),0)
\right\}.
\label{eq:two-sided-PL}
\end{equation}
\end{proposition}

\begin{proof}
The definition of $\lambda_{\min}$ implies that $f$ is $\lambda_{\min}$-strongly convex on $\mathcal K_0$. Hence the standard PL inequality gives
\begin{equation}\label{eq:PL}
f(v)-f(v^\star)
\leq
\frac{1}{2\lambda_{\min}}
\|\nabla f(v)\|^2.
\end{equation}
Set $\chi:=\nabla f(v)=c(v)-b$. For every point $\xi$ on the segment joining $0$ and $\chi$, each coordinate of $b+\xi$ lies between $b_j$ and $c_j(v)$. Therefore,
\begin{equation*}
\nabla^2\phi^*(\xi)
=
\operatorname{diag}(b+\xi)^{-1}
\succeq
\frac1{M_c}I.
\end{equation*}
Taylor's theorem gives points $\widetilde\xi$ and $\widehat\xi$ on this segment such that
\begin{equation*}
\begin{aligned}
D_{\phi^*}(0,\chi)
&=
\frac12\chi^\top\nabla^2\phi^*(\widetilde\xi)\chi
\geq
\frac{1}{2M_c}\|\chi\|^2,\\
\text{and } \qquad \qquad D_{\phi^*}(\chi,0)
&=
\frac12\chi^\top\nabla^2\phi^*(\widehat\xi)\chi
\geq
\frac{1}{2M_c}\|\chi\|^2.
\end{aligned}
\end{equation*}
Combining either estimate with the PL inequality \eqref{eq:PL} yields the desired result.
\end{proof}

By Bregman duality, the reverse-KL part of \eqref{eq:two-sided-PL} can be written as
\begin{equation}
\mu D_{f^*}(0,\nabla f(v))\leq D_{\phi^*}(0,\nabla f(v)),\qquad v\in\mathcal K_0.
\end{equation}
This bound compares divergences from the minimizer on $\mathcal K_0$; it does not assert relative strong convexity for arbitrary pairs. The constant depends on the problem and $\mathcal K_0$.

\begin{remark}[Sharp local rates for Sinkhorn]
Gao, Qu, Ye, and Udell~\cite{gao2026tight} give a nonasymptotic local analysis of Sinkhorn--Knopp that matches its asymptotic spectral rate, and develop accelerated matrix-scaling methods. Their rate is governed by the spectral gap of the optimally scaled matrix. Proposition~\ref{prop:pl} instead provides a bound on the initial sublevel set, sufficient for the global descent argument below; it need not give the sharp local rate.
\end{remark}

\subsection{Convergence of Sinkhorn}
\label{sec:sinkhorn-convergence}

Dual relative smoothness gives a one-step descent estimate for Sinkhorn. Combined with the sublevel-set bound of Proposition~\ref{prop:pl}, it yields global linear convergence and the sufficient decrease used later.

\begin{lemma}[One-step descent]
\label{lem:sinkhorn-descent}
Let
\begin{equation*}
\Smap(v):=\pi_{\mathsf V}\bigl(v-\nabla\phi^*(\nabla f(v))\bigr)
\end{equation*}
be one iteration of the normalized Sinkhorn map.
Then
\begin{equation}
\begin{aligned}
f(\Smap(v))-f(v)
&\leq
-D_{\phi^*}(0,\nabla f(v))
-D_{\phi^*}(\nabla f(\Smap(v)),0).
\end{aligned}
\label{eq:sinkhorn-descent}
\end{equation}
\end{lemma}

\begin{proof}
Set $v^+=\Smap(v)$. Since $\nabla f(v),\nabla f(v^+)\in \mathsf V$, the projection does not change their inner products. Hence
\begin{equation*}
\left\langle\nabla f(v^+),v-v^+\right\rangle
=
\left\langle
\nabla f(v^+),
\nabla\phi^*(\nabla f(v))-\nabla\phi^*(0)
\right\rangle.
\end{equation*}
Expanding $f$ at $v^+$ and using Bregman duality gives
\begin{equation*}
\begin{aligned}
f(v^+)-f(v)
={}&
-\left\langle
\nabla f(v^+),
\nabla\phi^*(\nabla f(v))-\nabla\phi^*(0)
\right\rangle
-D_{f^*}\bigl(\nabla f(v^+),\nabla f(v)\bigr).
\end{aligned}
\end{equation*}
The three-point identity for $D_{\phi^*}$ gives
\begin{equation*}
\begin{aligned}
f(v^+)-f(v)
={}&
-D_{\phi^*}(\nabla f(v^+),0)
-D_{\phi^*}(0,\nabla f(v))\\
&\quad
+D_{\phi^*}\bigl(\nabla f(v^+),\nabla f(v)\bigr)
-D_{f^*}\bigl(\nabla f(v^+),\nabla f(v)\bigr).
\end{aligned}
\end{equation*}
The sum of the last two terms is nonpositive by dual relative smoothness \eqref{eq:dual-relative-smoothness-bregman}.
\end{proof}

\begin{lemma}[Forward-KL descent]
\label{lem:sinkhorn-kl-descent}
For $v\in\mathsf V$,
\begin{equation}
\KL(c(\Smap(v))\|b)-\KL(c(v)\|b)\leq-D_f(\Smap(v),v).
\label{eq:sinkhorn-kl-descent}
\end{equation}
\end{lemma}
\begin{proof}
Set $v^+=\Smap(v)$. Expand $\phi^*$ at $\nabla f(v)$ and use $v-v^+=\pi_{\mathsf V}\nabla\phi^*(\nabla f(v))$. Since $\nabla f(v^+)-\nabla f(v)\in\mathsf V$, the projection does not change the inner product, and
\begin{equation*}
\begin{aligned}
\KL(c(v^+)\|b)-\KL(c(v)\|b)
&=\langle v-v^+,\nabla f(v^+)-\nabla f(v)\rangle+D_{\phi^*}(\nabla f(v^+),\nabla f(v))\\
&=-D_f(v^+,v)-D_f(v,v^+)+D_{\phi^*}(\nabla f(v^+),\nabla f(v)) \leq-D_f(v^+,v),
\end{aligned}
\end{equation*}
where the last step uses \eqref{eq:dual-relative-smoothness-bregman} and Bregman duality.
\end{proof}

\begin{lemma}[Last-iterate bound; Theorem~3.9 in \cite{maddison2021dual}]
\label{lem:sinkhorn-last-iterate}
Let $v^{k+1}=\Smap(v^k)$ with $v^0\in\mathsf V$. Then
\begin{equation}
\operatorname{KL}(c(v^k)\|b)
=
D_{\phi^*}(\nabla f(v^k),0)
\leq
\frac{f(v^0)-f(v^\star)}{k},
\qquad
k\geq1.
\label{eq:sinkhorn-last-iterate}
\end{equation}
\end{lemma}
\begin{proof}
By Lemma~\ref{lem:sinkhorn-kl-descent}, the forward KL is nonincreasing. Summing \eqref{eq:sinkhorn-descent} from $j=0$ to $k-1$ and dropping the reverse-KL terms gives
\begin{equation*}
k\KL(c(v^k)\|b)\leq\sum_{j=1}^k\KL(c(v^j)\|b)\leq f(v^0)-f(v^k)\leq f(v^0)-f(v^\star).
\end{equation*}
Dividing by $k$ proves the result.
\end{proof}

After fixing the gauge, coercivity and the PL inequality on the initial sublevel set yield a global linear rate for the reduced dual objective. The rate constant, however, can be poor when $\varepsilon$ is small~\cite{carlier2022linear,aubin2022mirror}.

\begin{proposition}[Linear convergence of normalized Sinkhorn]
\label{prop:sinkhorn-linear-rate}
Let $\{v^k\}_{k\geq0}$ be generated by the normalized Sinkhorn iteration, and let
$\mu
=
\lambda_{\min}/M_c
$
be the constant defined in \eqref{eq:consts}. 
Then, for every $k\geq0$,
\begin{equation}
f(v^{k+1})-f(v^\star)
\leq
\frac{1-\mu}{1+\mu}
\bigl(f(v^k)-f(v^\star)\bigr).
\label{eq:sinkhorn-linear-rate}
\end{equation}
\end{proposition}

\begin{proof}
Lemma~\ref{lem:sinkhorn-descent} shows that $f(v^{k+1})\leq f(v^k)$, so every iterate belongs to $\mathcal K_0$. Set
\begin{equation*}
E_k:=f(v^k)-f(v^\star).
\end{equation*}
Combining \eqref{eq:sinkhorn-descent} with the two-sided PL estimate \eqref{eq:two-sided-PL} gives
\begin{equation*}
\begin{aligned}
E_{k+1}-E_k
&\leq
-D_{\phi^*}(0,\nabla f(v^k))
-D_{\phi^*}(\nabla f(v^{k+1}),0) \leq
-\mu E_k-\mu E_{k+1}.
\end{aligned}
\end{equation*}
Rearranging this inequality proves \eqref{eq:sinkhorn-linear-rate}.
\end{proof}%

\section{Acc-Sinkhorn}
\label{sec:algorithm}

We insert the Sinkhorn displacement into a two-state HNAG scheme, enforce sufficient decrease by restarting, and decrease the coupling parameter through a homotopy schedule. 

\subsection{Accelerated Sinkhorn}
HNAG~\cite{chen2019orderoptimizationmethodsbased} is built on the first-order system
\begin{equation*}
x'=y-x-\beta\nabla f(x),
\qquad
y'=x-y-\frac1\mu\nabla f(x).
\end{equation*}
Eliminating $y$ gives a heavy-ball equation with the damping term $\beta\nabla^2f(x)x'$, which is the Hessian-driven damping of~\cite{AttouchChbaniFadiliRiahi2022First-order}. If $f$ is $\mu$-strongly convex, the energy $f(x)-f(x^\star)+\frac\mu2\|y-x^\star\|^2$ decays like $e^{-t}$, and a semi-implicit discretization with step $\alpha$ of order $\sqrt{\mu/L}$ contracts it by the factor $1/(1+\alpha)$ for $L$-smooth $f$~\cite{chen2019orderoptimizationmethodsbased,ChenHaoWei2025Accelerated}. We replace the unknown curvature constant by a parameter $\gamma$. 

The Sinkhorn direction is $\nabla\phi^*(\nabla f(x))=\log(c(x)./b)$. In the gauge $\mathsf V$, we project this direction onto $\mathsf V$ and use it in place of the Euclidean gradient:
\begin{equation}
\begin{aligned}
\frac{x_{k+1}-x_k}{\alpha}
&=
y_k-x_{k+1}-\beta \pi_{\mathsf V}\nabla\phi^*(\nabla f(x_k)),\\
\frac{y_{k+1}-y_k}{\alpha}
&=
x_{k+1}-y_{k+1}-\frac1\gamma \pi_{\mathsf V}\nabla\phi^*(\nabla f(x_{k+1})).
\end{aligned}
\label{eq:acc-sinkhorn-scheme}
\end{equation}
Here $x_k\in\mathsf V$ is the main dual potential, $y_k\in\mathsf V$ is an auxiliary state, and $\gamma>0$ is a free algorithmic parameter. As in HNAG$^{++}$~\cite{chen2025hnag}, we set
\begin{equation*}
\alpha=\sqrt{2\gamma},
\qquad
\alpha\beta=1.
\end{equation*}
Using $\pi_{\mathsf V}\nabla\phi^*(\nabla f(x))=x-\Smap(x)$ for $x\in\mathsf V$, one step of \eqref{eq:acc-sinkhorn-scheme} reads
\begin{equation}
x_{k+1}
=
\frac{\alpha y_k+\Smap(x_k)}{1+\alpha},
\qquad
y_{k+1}
=
\frac{\alpha y_k+2(\gamma-1)x_{k+1}+2\Smap(x_{k+1})}{\alpha(1+\alpha)}.
\label{eq:acc-sinkhorn-two-state}
\end{equation}
The value $\Smap(x_{k+1})$ in the second update is stored and reused in the next $x$-update.

\subsection{Descent test and restart}
A momentum step need not decrease $f$. At small $\varepsilon$, the unchecked iteration is unstable on some tested instances (Section~\ref{sec:synthetic}). We therefore accept a momentum step only if it decreases $f$ sufficiently. 

If the test fails, the method takes one Sinkhorn step and resets the auxiliary state to the new point, as in restart schemes for accelerated methods~\cite{odonoghue2015adaptive}. Since $x_0,y_0\in\mathsf V$ and $\Smap(\mathsf V)\subset\mathsf V$, induction gives $x_k,y_k,s_k\in\mathsf V$ for every $k$. 

\begin{algorithm}
\caption{Fixed-parameter Acc-Sinkhorn}
\label{alg:fixed-acc}
\begin{minipage}{0.94\linewidth}
\textbf{Input:} $(x_0,y_0)\in\mathsf V\times\mathsf V$, $0<\gamma\leq1$, a number of momentum steps $M\geq1$, and a stopping rule.
\begin{enumerate}[label=\arabic*.,leftmargin=2em,itemsep=2pt,topsep=4pt]
\item If the stopping rule is satisfied at $x_0$, return $(x_0,y_0)$.
\item Set $\alpha=\sqrt{2\gamma}$, $\sigma=1/(1+\sqrt2)$, $s_0=\Smap(x_0)$, and $k=0$.
\item Repeat until $M$ steps have been accepted:
\begin{enumerate}[label=(\roman*),leftmargin=1.5em]
\item \emph{Trial.} Compute 
\begin{equation*}
x_{k+1}=\frac{\alpha y_k+s_k}{1+\alpha}
\end{equation*} 
and the row solve $u(x_{k+1})$, which gives $f(x_{k+1})=1-\langle a,u(x_{k+1})\rangle-\langle b,x_{k+1}\rangle$.
\item \emph{Test.} If $f(x_{k+1})\leq f(x_k)-\sigma\KL(b\|c(x_k))$, complete $s_{k+1}=\Smap(x_{k+1})$ by the column reduction and set
\begin{equation*}
y_{k+1}=\frac{\alpha y_k+2(\gamma-1)x_{k+1}+2s_{k+1}}{\alpha(1+\alpha)}.
\end{equation*}
Count this as an accepted momentum step.
\item \emph{Restart.} Otherwise overwrite $x_{k+1}=s_k$, set $s_{k+1}=\Smap(x_{k+1})$ and $y_{k+1}=x_{k+1}$, and do not count this iteration toward $M$.
\item Increase $k$. If the stopping rule is satisfied at $x_k$, return $(x_k,y_k)$.
\end{enumerate}
\item Return $(x_k,y_k)$.
\end{enumerate}
\end{minipage}
\end{algorithm}

The test uses $f(x_{k+1})$, obtained from the trial row solve, and the cached value $\KL(b\|c(x_k))$. If it passes, the column reduction completes one Sinkhorn map, so the test adds no matrix pass. If it fails, the rejected row solve and the map at the restart point cost three reductions, equivalent to one and a half Sinkhorn maps. Each reduction costs $\mathcal O(mn)$, and the remaining vector operations cost $\mathcal O(m+n)$.

\begin{lemma}[No consecutive restarts]
\label{lem:rs-no-consecutive}
With $\sigma=1/(1+\sqrt2)$ and $0<\gamma\leq1$, no two consecutive iterations of Algorithm~\ref{alg:fixed-acc} are restarts. The algorithm returns after at most $2M+1$ iterations.
\end{lemma}

\begin{proof}
After the restart, $y_{k+1}=x_{k+1}$, so the next trial is $x_{k+2}=(\alpha x_{k+1}+s_{k+1})/(1+\alpha)$ with $s_{k+1}=\Smap(x_{k+1})$. By convexity of $f$ and Lemmas~\ref{lem:mirror-Bregman-KL} and \ref{lem:sinkhorn-descent},
\begin{equation*}
\begin{aligned}
f(x_{k+2})
\leq
\frac{\alpha f(x_{k+1})+f(s_{k+1})}{1+\alpha}
&\leq
f(x_{k+1})-\sigma \KL(b\|c(x_{k+1})),
\end{aligned}
\end{equation*}
The inequality uses $\sigma=1/(1+\sqrt2)\leq1/(1+\alpha)$, so the test passes. There are $M$ accepted steps, and between restarts there is at least one accepted step, so there are at most $M+1$ restarts.
\end{proof}

Among $k$ iterations, at most $\lceil k/2\rceil$ are restarts. Excluding initialization, their total cost is at most $k+\frac12\lceil k/2\rceil$ Sinkhorn maps, or $1.25$ maps per iteration up to an endpoint correction. Restarts are infrequent in the reported experiments (Section~\ref{sec:synthetic}).

\subsection{Homotopy}
A run of Algorithm~\ref{alg:fixed-acc} inside Algorithm~\ref{alg:homotopy-acc} is called a stage; its steps are the accepted steps, and restarts are counted separately. 
To avoid estimating the convexity constant, we start with $\gamma_0=1$ and halve $\gamma$ between stages.

\begin{algorithm}
\caption{Acc-Sinkhorn with homotopy}
\label{alg:homotopy-acc}
\label{alg:restart}
\begin{minipage}{0.94\linewidth}
\textbf{Input:} $x_0\in \mathsf V$, and a stopping rule.
\begin{enumerate}[label=\arabic*.,leftmargin=2em,itemsep=2pt,topsep=4pt]
\item Set $y_0 = x_0$, $(x,y)=(x_0,y_0)$, $\gamma_0=1$, and $N_0=0$.
\item For $\ell=0,1,\ldots$, set
\begin{equation*}
\gamma_\ell=2^{-\ell}\gamma_0,
\qquad
\alpha_\ell=\sqrt{2\gamma_\ell},
\qquad
M_\ell=\left\lceil\frac{\log 4}{\log(1+\alpha_\ell)}\right\rceil.
\end{equation*}
\item Run Algorithm~\ref{alg:fixed-acc} from $(x,y)$ with parameter $\gamma_\ell$, budget $M_\ell$, and the stopping rule. If the stopping rule is met, return its output; otherwise, replace $(x,y)$ by the output.
\item Let $N_{\ell+1}$ be the number of iterations, accepted steps and restarts, performed so far, and continue.
\end{enumerate}
\end{minipage}
\end{algorithm}

\section{Convergence}
\label{sec:restart}
\label{sec:convergence}

We prove global linear convergence of the safeguarded method and local accelerated linear convergence of the momentum iteration at a fixed parameter. Appendix~\ref{app:local-convergence} gives a global $\mathcal O(k^{-2})$ estimate under a uniform bound on restarts per stage. 

\subsection{Global convergence}
\label{sec:restart-global}

Both the accepted momentum step and the restart satisfy the same sufficient decrease. Combining this decrease with the global error bound gives a uniform contraction of the objective gap.

\begin{theorem}[Global linear convergence]
\label{thm:rs-global}
Let $\{x_k\}$ be the iterates, accepted steps and restarts, of Algorithm~\ref{alg:homotopy-acc} starting from any $x_0\in\mathsf V$. Then $x_k\in\mathcal K_0$ for all $k$, and
\begin{equation}
f(x_{k+1})-f(x^\star)
\leq
(1-\sigma\mu)\bigl(f(x_k)-f(x^\star)\bigr).
\label{eq:rs-global-rate}
\end{equation}
Consequently,
\begin{equation*}
\max\left\{
\KL(b\|c(x_k)),
\frac{\lambda_{\min}}{2}\|x_k-x^\star\|^2
\right\}
\leq
(1-\sigma\mu)^k\bigl(f(x_0)-f(x^\star)\bigr).
\end{equation*}%
\end{theorem}

\begin{proof}
Every iteration satisfies $f(x_{k+1})\leq f(x_k)-\sigma \KL(b\|c(x_k))$: an accepted step by the test, and a restart by Lemma~\ref{lem:sinkhorn-descent} and $\sigma\leq1$. Hence $f(x_k)$ is nonincreasing and $x_k\in\mathcal K_0$. Proposition~\ref{prop:pl} gives $\KL(b\|c(x_k))\geq\mu(f(x_k)-f(x^\star))$, which proves \eqref{eq:rs-global-rate}. The marginal bound follows from Lemma~\ref{lem:residual-from-gap}, and the iterate bound from $\lambda_{\min}$-strong convexity on $\mathcal K_0$.
\end{proof}

The estimate requires no condition on the auxiliary state, metric changes, or the number of restarts. By Lemma~\ref{lem:rs-no-consecutive}, a stage with $M_s$ accepted momentum steps takes at most $2M_s+1$ iterations. A rejected trial costs one additional row reduction.

\subsection{Local accelerated convergence}
\label{sec:conditional-convergence}
\label{sec:local-accelerated}
\label{sec:acc-local}
For fixed $\gamma$, the momentum iteration can be analyzed by linearization at $(x^\star,x^\star)$. Define
\begin{equation}
\mu_\star:=\lambda_{\min}\bigl(\pi_{\mathsf V}\diag(b)^{-1}\nabla^2f(x^\star)|_{\mathsf V}\bigr)\in(0,1].
\label{eq:loc-mu}
\end{equation}
The constant $\mu_\star$ measures the local curvature in the Sinkhorn geometry.

\begin{theorem}[Local accelerated linear convergence]
\label{thm:acc-local}
For the momentum iteration \eqref{eq:acc-sinkhorn-scheme} without restarts, with fixed $\gamma>0$, $\alpha=\sqrt{2\gamma}$, and $\alpha\beta=1$, the solution $(x^\star,x^\star)$ is locally attracting. If $0<\gamma\leq\mu_\star$, then for every $\rho\in((1+\sqrt{2\gamma})^{-1},1)$, initial pairs $(x_0,y_0)\in\mathsf V\times\mathsf V$ sufficiently close to $(x^\star,x^\star)$ satisfy
\begin{equation}
\|x_k-x^\star\|_2+\|y_k-x^\star\|_2\leq C\rho^k\bigl(\|x_0-x^\star\|_2+\|y_0-x^\star\|_2\bigr),
\label{eq:acc-local-rate}
\end{equation}
and both objective gaps are $\mathcal O(\rho^{2k})$.
\end{theorem}
\begin{proof}[Proof of Theorem~\ref{thm:acc-local}]
Set $g(x)=\pi_{\mathsf V}\nabla\phi^*(\nabla f(x))$ and $\alpha^2=2\gamma$. The updates are
\begin{equation*}
(1+\alpha)x_{k+1}=x_k+\alpha y_k-g(x_k),
\qquad
(1+\alpha)y_{k+1}=y_k+\alpha x_{k+1}-\frac2\alpha g(x_{k+1}).
\end{equation*}
Since $\nabla f(x^\star)=0$ and $\nabla\phi^*(0)=0$, we have $g(x^\star)=0$. Subtracting the same equations at $(x^\star,x^\star)$ gives the exact error equations
\begin{equation*}
\begin{aligned}
(1+\alpha)(x_{k+1}-x^\star)
&=(x_k-x^\star)+\alpha(y_k-x^\star)-[g(x_k)-g(x^\star)],\\
(1+\alpha)(y_{k+1}-x^\star)
&=(y_k-x^\star)+\alpha(x_{k+1}-x^\star)-\frac2\alpha[g(x_{k+1})-g(x^\star)].
\end{aligned}
\end{equation*}

Set $H:=\nabla^2f(x^\star)|_{\mathsf V}$. Differentiation gives
\begin{equation*}
\nabla g(x^\star)=\left.\pi_{\mathsf V}\diag(b)^{-1}\right|_{\mathsf V}H.
\end{equation*}
Thus $\nabla g(x^\star)$ is self-adjoint and positive definite in the inner product $\langle z,w\rangle_H=z^\top Hw$. Proposition~\ref{prop:dual-relative-smoothness} gives $\left.\pi_{\mathsf V}\diag(b)^{-1}\right|_{\mathsf V}\preceq H^{-1}$, so its eigenvalues lie in $[\mu_\star,1]$. Choose an $H$-orthonormal eigenbasis. An eigenmode is a vector $z$ in this basis with $\nabla g(x^\star)z=\lambda z$. 

Let $\xi_k$ and $\eta_k$ be the coefficients of the linearized errors $x_k-x^\star$ and $y_k-x^\star$ in this eigenmode. The error equations reduce to a triangular system:
\begin{equation*}
\begin{pmatrix}1+\alpha&0\\-(\alpha-2\lambda/\alpha)&1+\alpha\end{pmatrix}
\begin{pmatrix}\xi_{k+1}\\\eta_{k+1}\end{pmatrix}
=
\begin{pmatrix}1-\lambda&\alpha\\0&1\end{pmatrix}
\begin{pmatrix}\xi_k\\\eta_k\end{pmatrix}.
\end{equation*}
Hence the Jacobian block is
\begin{equation*}
\begin{aligned}
M_\lambda &=
\begin{pmatrix}1+\alpha&0\\-(\alpha-2\lambda/\alpha)&1+\alpha\end{pmatrix}^{-1}
\begin{pmatrix}1-\lambda&\alpha\\0&1\end{pmatrix}.
\end{aligned}
\end{equation*}

The determinant is the product of the two eigenvalues. From the triangular factorization,
\begin{equation*}
\det M_\lambda=\frac{1-\lambda}{(1+\alpha)^2},
\qquad
0\leq\det M_\lambda\leq \frac{1 - \mu_{\star}}{(1+\alpha)^2},
\end{equation*}
since $0<\lambda\leq1$. The determinant is zero only when $\lambda=1$.

We now treat $\alpha>0$ as the parameter. The two eigenvalues are
\begin{equation*}
\frac{\alpha^2+2\alpha+2-(\alpha+3)\lambda\ \pm\sqrt{\Delta_\lambda}}{2(1+\alpha)^2},
\quad
\Delta_\lambda:=\bigl[\alpha^2+2\alpha+2-(\alpha+3)\lambda\bigr]^2-4(1+\alpha)^2(1-\lambda).
\end{equation*}
They are a complex conjugate pair when $\Delta_\lambda<0$.
In this range, their product equals their squared modulus, so each has modulus 
\begin{equation*}
\sqrt{1-\lambda}/(1+\alpha)<(1+\alpha)^{-1}<1.
\end{equation*} 

Outside this range the eigenvalues are real, with a repeated root at equality.
For the real case, let $p_\lambda(z):=\det(zI-M_\lambda)$. Direct substitution gives
\begin{equation*}
\begin{aligned}
p_\lambda(1)&=\frac{(\alpha+2)\lambda}{(1+\alpha)^2}>0,\qquad 
p_\lambda(-1)&=\frac{2(1+\alpha)^2+2-(\alpha+4)\lambda}{(1+\alpha)^2}
\geq\frac{\alpha(2\alpha+3)}{(1+\alpha)^2}>0.
\end{aligned}
\end{equation*}
Since the root product is in $[0,1)$, these inequalities place both real roots in $(-1,1)$. Thus the spectral radius of $M_{\lambda}$ is strictly less than $1$. 

If $\alpha^2\leq2\lambda$, then
\begin{equation*}
\begin{aligned}
p_\lambda\bigl((1+\alpha)^{-1}\bigr)&=\frac{2\lambda-\alpha^2}{(1+\alpha)^3}\geq0,\qquad
p_\lambda\bigl(-(1+\alpha)^{-1}\bigr)&=\frac{(\alpha+2)(\alpha+2-2\lambda)}{(1+\alpha)^3}>0.
\end{aligned}
\end{equation*}
The root product is strictly less than $(1+\alpha)^{-2}$, so both real roots lie in $[-(1+\alpha)^{-1},(1+\alpha)^{-1}]$. Together with the complex case, this proves that the full Jacobian $M$ has spectral radius at most $(1+\alpha)^{-1}$ when $\alpha\leq\sqrt{2\mu_\star}$, equivalently $\gamma\leq\mu_\star$.

To pass to the nonlinear iteration, write $e_k=(x_k-x^\star,y_k-x^\star)$. Smoothness of the update gives $e_{k+1}=Me_k+\mathcal O(\|e_k\|_2^2)$. Given $(1+\alpha)^{-1}<\rho<1$, choose $q$ strictly between the spectral radius of $M$ and $\rho$. There is an equivalent norm $\|\cdot\|_*$ with $\|M\|_*\leq q$. In a sufficiently small ball, the remainder satisfies
$\|e_{k+1}-Me_k\|_*\leq C\|e_k\|_*^2.$
Choose the ball radius $\delta$ so that $C\delta\leq\rho-q$. Then
\begin{equation*}
\|e_{k+1}\|_*\leq(q+C\delta)\|e_k\|_*\leq\rho\|e_k\|_*.
\end{equation*}
The ball is invariant, so induction and norm equivalence give \eqref{eq:acc-local-rate}. For arbitrary fixed $\gamma>0$, the same argument with any $\rho$ between the spectral radius and one proves local attraction. Finally, $\nabla f(x^\star)=0$ and Taylor's theorem give $f(x_k)-f(x^\star)=\mathcal O(\|x_k-x^\star\|_2^2)$, and likewise for $y_k$, proving the objective estimates.
\end{proof}

Theorem~\ref{thm:acc-local} gives the accelerated factor $(1+\sqrt{2\gamma})^{-1}$ for a fixed parameter. For $\gamma=\mu_\star$, this is $(1+\sqrt{2\mu_\star})^{-1}\approx 1-\sqrt{2\mu_\star}$, compared with the local factor $1-\mu_\star$ of Sinkhorn~\cite{knight2008sinkhorn,gao2026tight}. The homotopy avoids estimating $\mu_\star$; its global guarantee is Theorem~\ref{thm:rs-global}.

\subsection{Global homotopy estimate}
\label{sec:global-quadratic}
The homotopy also admits a stagewise energy estimate when the number of restarts per stage is uniformly bounded.

\begin{theorem}[Global $\mathcal O(k^{-2})$ rate]
\label{thm:glob-rate}
Assume that every stage of Algorithm~\ref{alg:homotopy-acc} has at most $F$ restarts. Then there is a constant $C>0$, depending on $F$, $x_0$, and the problem, such that
\begin{equation}
f\bigl(\Smap(x_k)\bigr)-f(x^\star)\leq\frac{C}{(1+k)^2},
\qquad k\geq0,
\label{eq:glob-rate}
\end{equation}
where $k$ counts both momentum steps and restarts.
\end{theorem}

The proof is given in Appendix~\ref{app:local-convergence}. It controls metric changes by the finite path length from Theorem~\ref{thm:rs-global}, then uses the homotopy to bound the energy by a constant times $\gamma_s$. The constant can depend on the global curvature.

\section{Numerical experiments}
\label{sec:experiments}

We compare the work needed to reach a common marginal tolerance and report application-specific accuracy separately. The experiments use Python 3.14.3, NumPy 2.4.2, and SciPy 1.17.1 on one thread of an Apple M4 processor.

\subsection{Methods and work metric}
\label{sec:methods-work}

We compare Acc-Sinkhorn in Algorithm~\ref{alg:homotopy-acc}, with $\sigma=1/(1+\sqrt2)$, with six published methods:
\begin{itemize}[leftmargin=2em]
\item {\it Sinkhorn}, the ordinary Sinkhorn--Knopp scaling iteration;
\item {\it adaptive overrelaxed Sinkhorn}, Algorithm~3 of Thibault et al.~\cite{thibault2021overrelaxed}, which extrapolates each half-step by an adaptively safeguarded factor;
\item {\it annealed Sinkhorn}, Algorithm~1 of Chizat~\cite{chizat2024annealed}, which raises the inverse temperature along the iteration;
\item {\it APDAGD}, the adaptive primal--dual accelerated gradient method of Dvurechensky, Gasnikov, and Kroshnin~\cite{dvurechensky2018computational};
\item {\it APDAMD}, the adaptive primal--dual accelerated mirror descent method of Lin, Ho, and Jordan~\cite{lin2019efficient}; and
\item {\it Scaling0}, the linear-coupling first-order matrix-scaling method of Allen-Zhu, Li, Oliveira, and Wigderson~\cite{allenzhu2017scaling}.
\end{itemize}
We implement all six baselines in a common driver, following their published algorithms. Parameter settings are given in Appendix~\ref{sec:experimental-details}.

Every method in the comparison is first order, and its leading cost is exponentiating an $n\times m$ array using the stable log-domain primitives of Section~\ref{sec:stable-evaluation}. A value-and-gradient evaluation, equivalently one Sinkhorn map, uses two such passes, one row-wise and one column-wise, and we count it as one {\it first-order oracle call}. A value evaluation uses one pass and counts as half a call. Under this unit, a restart of Acc-Sinkhorn costs one and a half calls.

The stopping criterion is the marginal violation
\begin{equation}
\mathcal R(P)
:=
\KL(a\|P\one_m)
+
\KL(b\|P^\top\one_n)
<
\texttt{tol}.
\label{eq:residual}
\end{equation}
For row-reduced methods, the first term vanishes because $u(v)$ solves the row subproblem exactly. Also, $\log(c(v)./b)$ is available from the column reduction in \eqref{eq:stable-c}, so evaluating $\KL(b\|c(v))$ needs no extra matrix pass. This log-domain evaluation avoids forming coupling entries that may underflow.

We use an absolute tolerance, independent of the initial residual. For row-reduced plans, $c(v)$ and $b$ are probability vectors, so Pinsker's inequality gives $\|c(v)-b\|_1\leq\sqrt{2\,\texttt{tol}}<1.5\times10^{-3}$ when $\texttt{tol}=10^{-6}$. For plans that do not have unit mass, \eqref{eq:residual} uses the generalized KL divergence defined above.

\subsection{Synthetic problems}
\label{sec:synthetic}

\paragraph{Instances}
Entries of $a$, $b$, and $C$ are drawn independently and uniformly from $[0,1]$; $a$ and $b$ are then normalized to unit mass. All solvers operate on the rescaled cost $\bar C=C/\varepsilon$ of Section~\ref{sec:reduced-formulation}. We fix $n=m=100$ and sweep $\varepsilon\in\{10^{-1},10^{-2},10^{-3},10^{-4}\}$. Every solver starts from the same initial plan, $P_0=P(u(0),0)$: the column potential is $v_0=0$, the origin of the gauge $\mathsf V$, and the row potential is the exact row solve $u(v_0)$. 
Each configuration is repeated on three random seeds ($11$, $23$, $47$), and all reported curves and counts are seed averages. A run stops when \eqref{eq:residual} holds with $\texttt{tol}=10^{-6}$, or after $200{,}000$ iterations.

\paragraph{Error measure}
At the target regularization, let $P=P(u(v),v)$ be a row-reduced plan. Lemma~\ref{lem:residual-from-gap} identifies its divergence from the optimal plan with the reduced dual gap. We plot the normalized divergence
\begin{equation}
\frac{\KL(P^\star\|P)}{\KL(P^\star\|P_0)}
=\frac{f(v)-f(v^\star)}{f(v_0)-f(v^\star)},
\qquad P=P(u(v),v).
\label{eq:kl-ratio}
\end{equation}
Here $P^\star=P(u(v^\star),v^\star)$ and $P_0=P(u(v_0),v_0)$. For other plans, the left-hand ratio remains the comparison measure, but the reduced-gap identity need not hold. The reference plan is computed by an Acc-Sinkhorn solve with marginal residual below $10^{-15}$.

\paragraph{Work metric}
Table~\ref{tab:synthetic-kl} reports the common call count defined in Section~\ref{sec:methods-work}, together with wall time. As detailed in Appendix~\ref{sec:experimental-details}, timings exclude the reference solve and the extra evaluations used only to record comparison quantities.

\begin{figure}[tbp]
\centering
\includegraphics[width=0.8\linewidth]{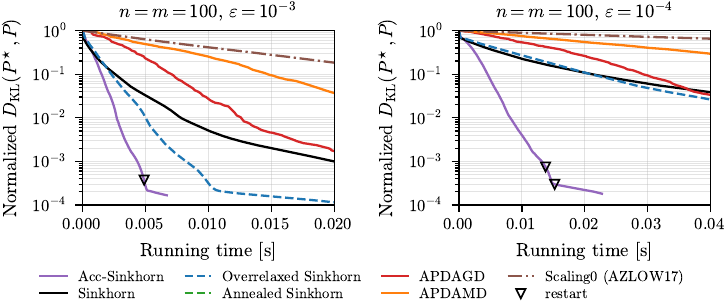}
\caption{Synthetic problems, $n=m=100$. Left: $\varepsilon=10^{-3}$; right: $10^{-4}$. Triangles mark Acc-Sinkhorn restarts on the first seed.
}\label{fig:synthetic-kl}
\end{figure}

\begin{table}[tbp]
\centering
\caption{Synthetic problems: oracle calls and wall time (seconds) to reach \eqref{eq:residual} with $\texttt{tol}=10^{-6}$, averaged over three seeds. Parentheses give mean Acc-Sinkhorn restarts. NR: target not reached within $200{,}000$ iterations. 
}
\label{tab:synthetic-kl}
\small
\renewcommand{\arraystretch}{1.125}
\setlength{\tabcolsep}{3pt}\resizebox{0.8\textwidth}{!}{
\begin{tabular}{@{}lrrrrrrrr@{}}
\toprule
& \multicolumn{2}{c}{$\varepsilon=10^{-1}$}
& \multicolumn{2}{c}{$\varepsilon=10^{-2}$}
& \multicolumn{2}{c}{$\varepsilon=10^{-3}$}
& \multicolumn{2}{c}{$\varepsilon=10^{-4}$}\\
\cmidrule(lr){2-3}\cmidrule(lr){4-5}\cmidrule(lr){6-7}\cmidrule(lr){8-9}
Method & Calls & Time & Calls & Time & Calls & Time & Calls & Time\\
\midrule
Sinkhorn & \textbf{3} & \textbf{0.0002} & 25 & 0.0020 & 460 & 0.038 & 7{,}693 & 0.585\\
Acc-Sinkhorn & 6 & 0.0005 & 15 & \textbf{0.0012} & \textbf{77} (0.67) & \textbf{0.0068} & \textbf{301} (2.33) & \textbf{0.025}\\
Overrelaxed Sinkhorn & \textbf{3} & 0.0006 & \textbf{11} & 0.0023 & 238 & 0.047 & 2{,}943 & 0.597\\
APDAGD & 331 & 0.030 & 455 & 0.044 & 1{,}298 & 0.142 & 4{,}259 & 0.396\\
APDAMD & 335 & 0.029 & 492 & 0.051 & 1{,}471 & 0.172 & 5{,}229 & 0.564\\
Scaling0 & 119 & 0.010 & 283 & 0.025 & 2{,}148 & 0.210 & 14{,}861 & 1.278\\
Annealed Sinkhorn & 4 & 0.0004 & 106 & 0.0116 & 9{,}997 & 0.893 & NR & --\\
\bottomrule
\end{tabular}
}
\end{table}

The reported Acc-Sinkhorn runs use $y_0=x_0$ and the same initial plan as the baselines. The descent test failed at most three times per run in Table~\ref{tab:synthetic-kl}.

\paragraph{Results}
At $\varepsilon\geq10^{-2}$, Sinkhorn, Acc-Sinkhorn, and overrelaxed Sinkhorn reach the target in at most $25$ calls. Annealed Sinkhorn needs $106$ calls at $\varepsilon=10^{-2}$; APDAGD, APDAMD, and Scaling0 require $119$--$492$ calls across these two regularization levels.

At $\varepsilon=10^{-3}$ and $10^{-4}$, Acc-Sinkhorn uses the fewest calls and has the shortest runtime. Its call count is lower than Sinkhorn's by factors of about $6$ and $26$, and lower than overrelaxed Sinkhorn's by factors of about $3$ and $10$, respectively. The corresponding runtime speedups over Sinkhorn are about $6$ and $24$.

\paragraph{An extreme case}
We test $\varepsilon=10^{-5}$ and compare Acc-Sinkhorn with the same iteration but without the descent test and restart. Both start at $x_0=y_0=0$ and otherwise run under the same protocol ($\texttt{tol}=10^{-6}$, budget $200{,}000$ iterations) as Table~\ref{tab:synthetic-kl}, on its three seeds and two additional seeds. %

\begin{table}[tbp]
\centering
\caption{Synthetic problems, $\varepsilon=10^{-5}$, $n=m=100$: oracle calls and wall time (seconds) per seed to reach $\mathcal R(P)<10^{-6}$. A restart costs $1.5$ calls; the last row gives Acc-Sinkhorn restarts. DIV: marginal violation increased sharply and failed to recover within the iteration budget.} %
\label{tab:eps1e-5}
\small
\renewcommand{\arraystretch}{1.125}
\setlength{\tabcolsep}{2.7pt}\resizebox{0.8\textwidth}{!}{
\begin{tabular}{@{}lrrrrrrrrrr@{}}
\toprule
& \multicolumn{2}{c}{Seed 5}
& \multicolumn{2}{c}{Seed 11}
& \multicolumn{2}{c}{Seed 23}
& \multicolumn{2}{c}{Seed 47}
& \multicolumn{2}{c}{Seed 89}\\
\cmidrule(lr){2-3}\cmidrule(lr){4-5}\cmidrule(lr){6-7}\cmidrule(lr){8-9}\cmidrule(lr){10-11}
Method & Calls & Time & Calls & Time & Calls & Time & Calls & Time & Calls & Time\\
\midrule
Without restart & \multicolumn{2}{c}{DIV} & 1{,}052 & 0.071 & \multicolumn{2}{c}{DIV} & \textbf{1{,}255} & \textbf{0.089} & \multicolumn{2}{c}{DIV}\\
Acc-Sinkhorn & \textbf{1{,}495.5} & \textbf{0.099} & \textbf{891.5} & \textbf{0.059} & \textbf{918} & \textbf{0.066} & 1{,}517.5 & 0.120 & \textbf{1{,}125} & \textbf{0.078}\\
\quad restarts & \multicolumn{2}{c}{13} & \multicolumn{2}{c}{7} & \multicolumn{2}{c}{8} & \multicolumn{2}{c}{7} & \multicolumn{2}{c}{10}\\
\bottomrule
\end{tabular}
}
\end{table}

On seeds 5, 23, and 89, the iteration without restarts drives the marginal residual up to order $10^5$ and does not recover within the budget. On seeds 11 and 47, it reaches the target; on seed 47, it uses fewer calls than Acc-Sinkhorn. Acc-Sinkhorn reaches the target on all five seeds, averaging $1{,}189.5$ calls and $0.084\,\mathrm{s}$, with seven to thirteen restarts per run. At larger regularization, restarts are rare, typically one or two or none. The extreme value $\varepsilon=10^{-5}$ serves as a stability test; the application experiments use larger regularization. These results are consistent with the global convergence guarantee in Theorem~\ref{thm:rs-global}.

\subsection{Color transfer}
\label{sec:color-transfer}

We next apply the solvers to color transfer~\cite{reinhard2001color,ferradans2014regularized}. A source image supplies the color palette, and a target image supplies the spatial content. We sample $n=1000$ RGB values from each image and normalize them to $[0,1]^3$. The cost is the normalized squared Euclidean color distance,
\begin{equation}
C_{ij}=\|x_i-y_j\|_2^2/\max_{r,s}\|x_r-y_s\|_2^2,
\qquad a_i=b_j=1/n.
\label{eq:color-cost}
\end{equation}
The source is a daylight ocean photograph and the target the same scene at sunset. All solvers start from the same plan $P(u(0),0)$, run under the work metric and instrumentation of Section~\ref{sec:methods-work}, and sweep $\varepsilon\in\{10^{-1},10^{-2},10^{-3},10^{-4}\}$, with the stopping threshold $\mathcal R(P)<10^{-6}$ of \eqref{eq:residual}. Each run is capped at $20{,}000$ iterations and at five minutes of wall time. %

\begin{table}[tbp]
\centering
\caption{Color transfer, $n=1000$: oracle calls and wall time (seconds) to reach $\mathcal R(P)<10^{-6}$. Parentheses give Acc-Sinkhorn restarts. T/O: five-minute limit reached before the iteration limit.}
\label{tab:color}
\small
\renewcommand{\arraystretch}{1.125}
\setlength{\tabcolsep}{2pt}\resizebox{0.7\textwidth}{!}{
\begin{tabular}{lrrrrrrrr}
\toprule
& \multicolumn{2}{c}{$\varepsilon=10^{-1}$}
& \multicolumn{2}{c}{$\varepsilon=10^{-2}$}
& \multicolumn{2}{c}{$\varepsilon=10^{-3}$}
& \multicolumn{2}{c}{$\varepsilon=10^{-4}$}\\
\cmidrule(lr){2-3}\cmidrule(lr){4-5}\cmidrule(lr){6-7}\cmidrule(lr){8-9}
Method & Calls & Time & Calls & Time & Calls & Time & Calls & Time\\
\midrule
Sinkhorn & \textbf{4} & 0.04 & 41 & 0.37 & 404 & 3.64 & 3{,}948 & 42.14\\
Acc-Sinkhorn & 6 & 0.05 & 17 & 0.15 & \textbf{51} & \textbf{0.44} & \textbf{163.5} (1) & \textbf{1.65}\\
Overrelaxed Sinkhorn & \textbf{4} & 0.04 & \textbf{14} & \textbf{0.13} & 74 & 0.69 & 548 & 5.86\\
APDAGD & 118 & 1.10 & 386 & 3.88 & 1{,}418 & 14.71 & 5{,}890 & 64.44\\
APDAMD & 768 & 8.13 & 1{,}288 & 12.95 & 2{,}260 & 22.87 & 8{,}496 & 101.78\\
Scaling0 & 88 & 0.79 & 394 & 3.55 & 2{,}316 & 20.24 & 19{,}972 & 213.11\\
Annealed Sinkhorn & 5 & 0.05 & 112 & 1.02 & 9{,}954 & 92.34 & T/O & --\\
\bottomrule
\end{tabular}
}
\end{table}

\paragraph{Cost at matched feasibility}
Table~\ref{tab:color} reports the work needed to reach $\mathcal R(P)<10^{-6}$. Acc-Sinkhorn is fastest at $\varepsilon=10^{-3}$ and $10^{-4}$. At $\varepsilon=10^{-4}$, it uses about $24$ times fewer calls than Sinkhorn and $3.4$ times fewer than overrelaxed Sinkhorn; the runtime speedups are about $26$ and $3.5$, respectively. Annealed Sinkhorn reaches the time limit before meeting the tolerance.

At $\varepsilon=10^{-4}$, Scaling0 uses about $122$ times as many calls as Acc-Sinkhorn and takes about $129$ times as long. Its prescribed diameter bound is $N\approx10^4$, giving an initial step $\tau_0=1/(32N)\approx3\times10^{-6}$. The table reports total work and does not separate initialization from subsequent iterations.

\begin{figure}[tbp]
\centering
\includegraphics[width=0.825\linewidth]{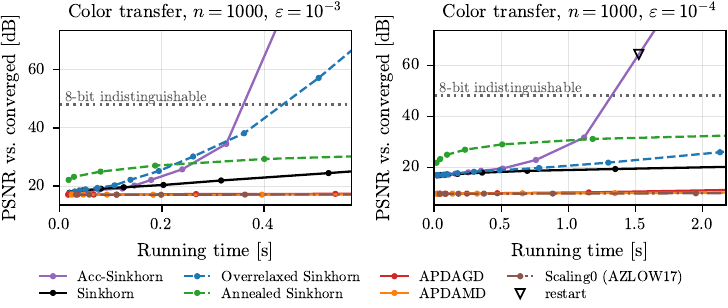}
\caption{Color transfer, $n=1000$: rendering PSNR versus wall time. Left: $\varepsilon=10^{-3}$; right: $10^{-4}$. 
The dotted line marks RMS error $1/255$ ($48.1$ dB), one $8$-bit intensity level. The triangle marks the single Acc-Sinkhorn restart at $\varepsilon=10^{-4}$.
}
\label{fig:color-psnr}
\end{figure}

\paragraph{Rendering accuracy}
We compare each rendered image with the rendering from a tightly converged reference plan at the same $\varepsilon$. With colors in $[0,1]$, the peak signal-to-noise ratio is $\mathrm{PSNR}=10\log_{10}(1/\mathrm{MSE})$, where MSE is the mean squared error over pixels and channels. The $48.1$ dB threshold corresponds to an RMS error of $1/255$, one $8$-bit intensity level. It does not guarantee identical quantized images or perceptual indistinguishability.

Figure~\ref{fig:color-psnr} reports PSNR against wall time. At $\varepsilon=10^{-4}$, Acc-Sinkhorn reaches $48.1$ dB after $140$ calls, compared with $163.5$ calls to reach $\mathcal R(P)<10^{-6}$. Thus the chosen rendering threshold is reached before the marginal tolerance in this example.

\begin{figure}[htp]
\centering
\begin{minipage}[b]{0.32\linewidth}
\centering
\includegraphics[width=0.85\linewidth]{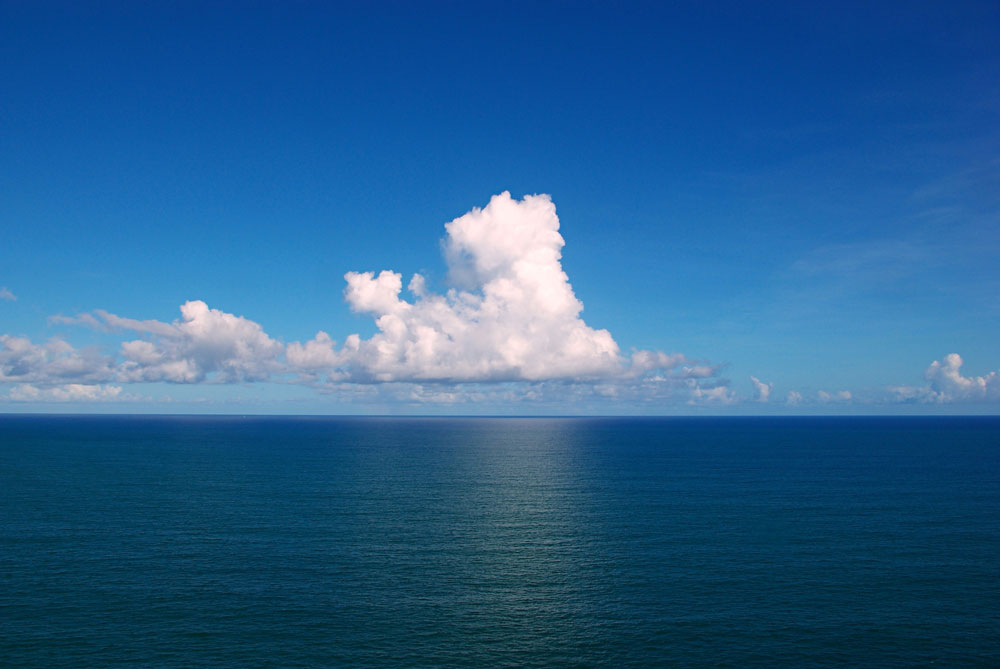}
\par\smallskip\small (a) Source palette
\end{minipage}
\hfill
\begin{minipage}[b]{0.32\linewidth}
\centering
\includegraphics[width=0.75\linewidth]{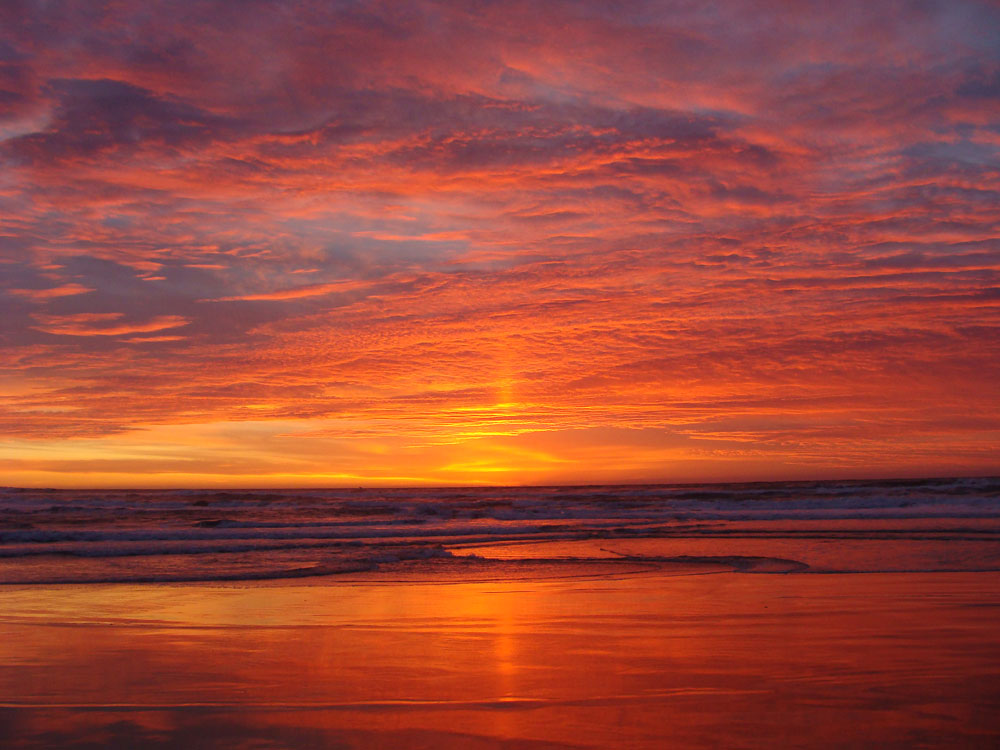}
\par\smallskip\small (b) Target content
\end{minipage}

\begin{minipage}[b]{0.18\linewidth}
\centering
\includegraphics[width=\linewidth]{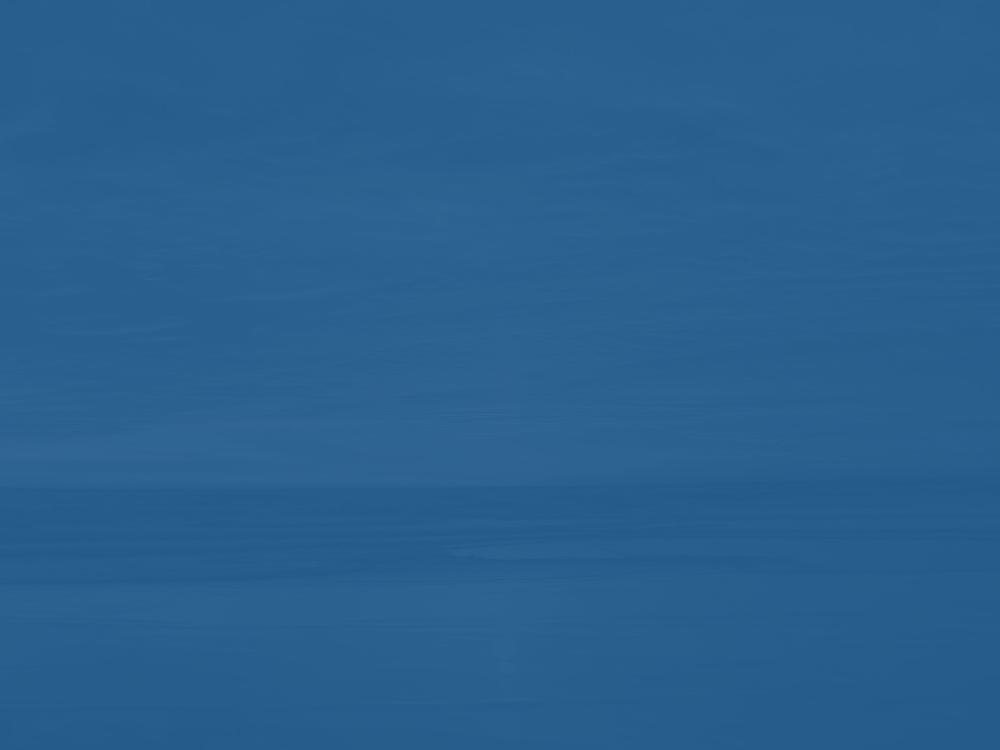}
\par\smallskip\small (c) $\varepsilon=1$
\end{minipage}
\hfill
\begin{minipage}[b]{0.18\linewidth}
\centering
\includegraphics[width=\linewidth]{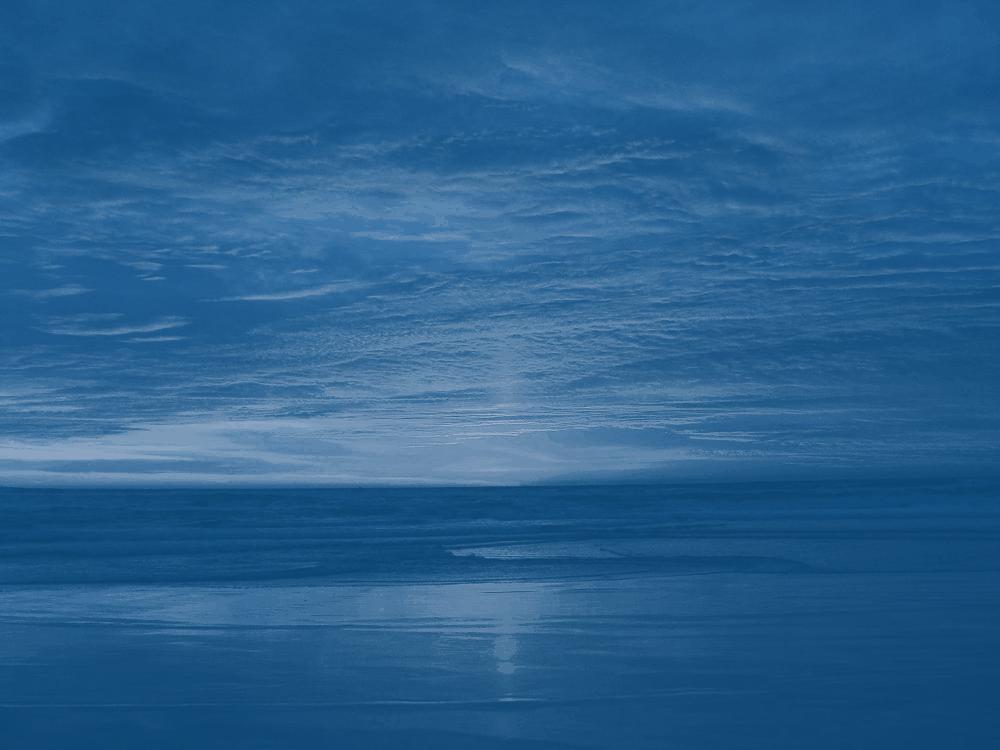}
\par\smallskip\small (d) $10^{-1}$
\end{minipage}
\hfill
\begin{minipage}[b]{0.18\linewidth}
\centering
\includegraphics[width=\linewidth]{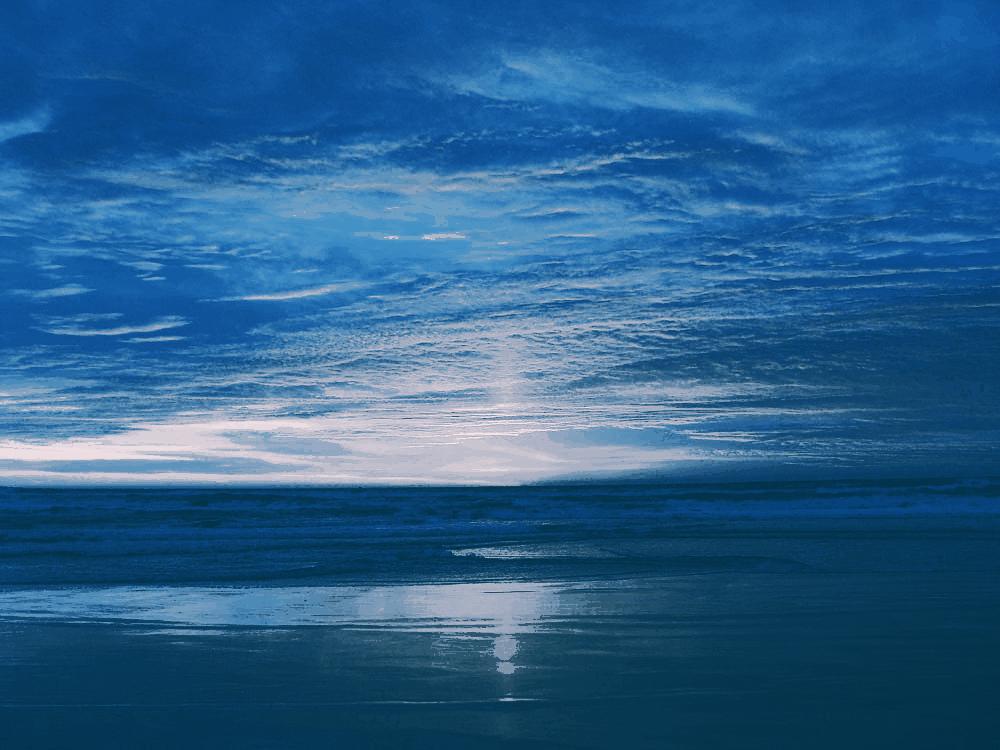}
\par\smallskip\small (e) $10^{-2}$
\end{minipage}
\hfill
\begin{minipage}[b]{0.18\linewidth}
\centering
\includegraphics[width=\linewidth]{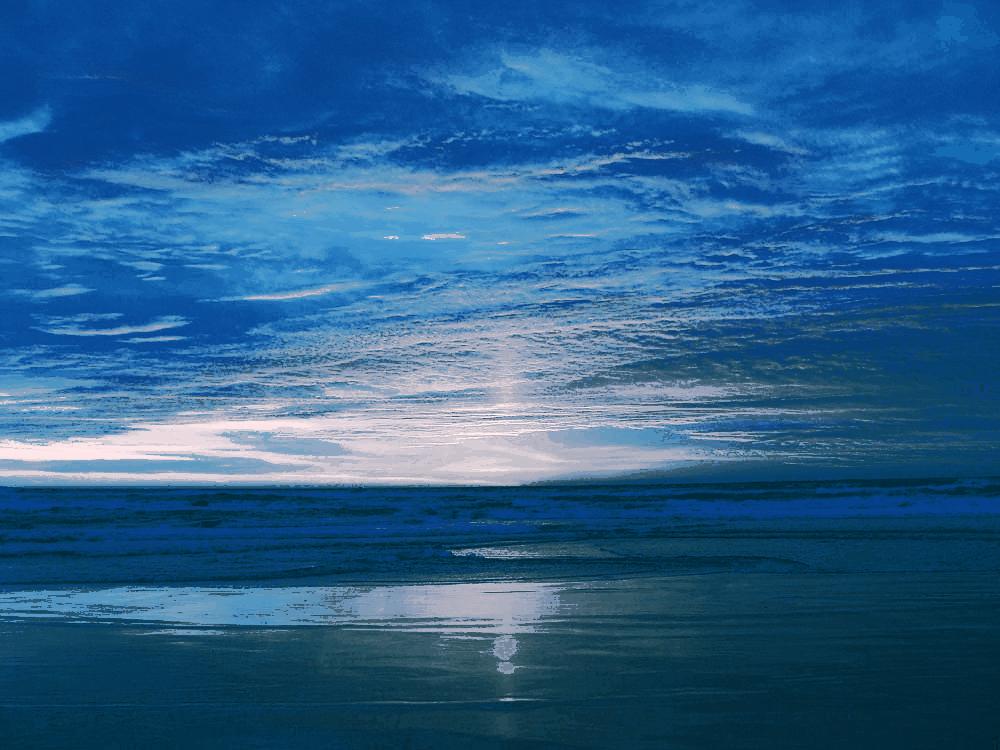}
\par\smallskip\small (f) $10^{-3}$
\end{minipage}
\hfill
\begin{minipage}[b]{0.18\linewidth}
\centering
\includegraphics[width=\linewidth]{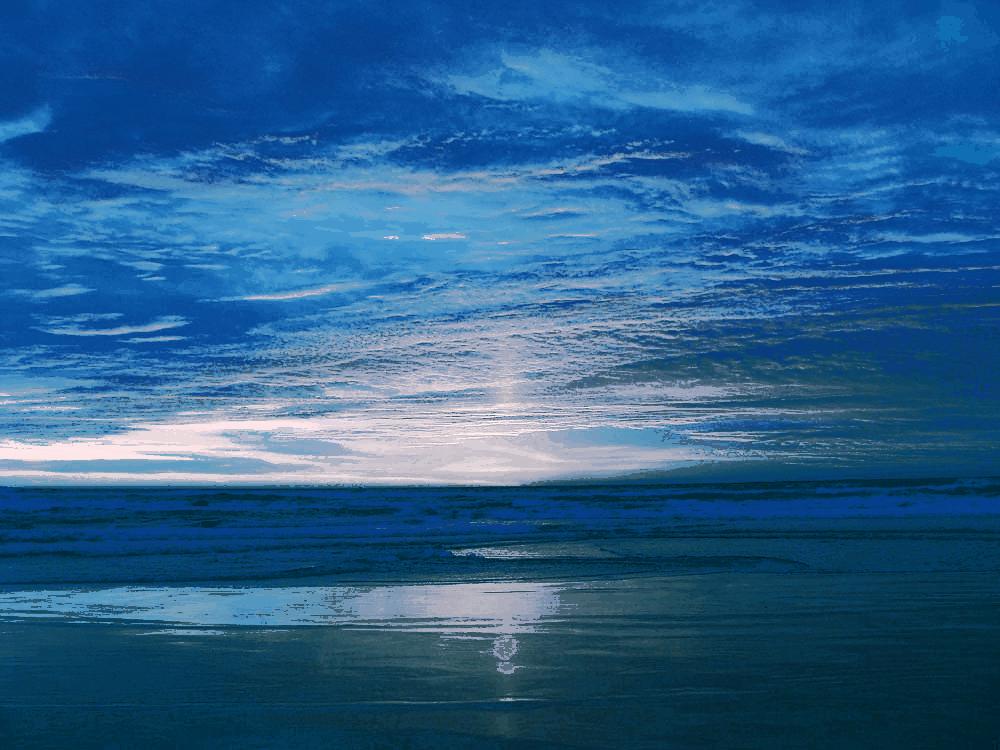}
\par\smallskip\small (g) $10^{-4}$
\end{minipage}
\caption{Acc-Sinkhorn color transfer: source palette and target content. Smaller $\varepsilon$ sharpens transported colors and makes the problem harder.}
\label{fig:color-transfer}
\end{figure}

The final plan is converted to colors by barycentric projection. Nearest-neighbor lookup extends the sampled colors to the full target image, and the result is clipped to $[0,1]^3$. Figure~\ref{fig:color-transfer} shows the effect of regularization. Smaller $\varepsilon$ produces sharper transported colors and requires more work.

\subsection{Bilingual word alignment}
\label{sec:word-alignment}

The final study tests whether faster optimization preserves downstream quality. We use aligned multilingual fastText vectors~\cite{conneau2018word,bojanowski2017enriching}. The top $n=500$ English and French vectors are normalized in $\R^{300}$. Ground-truth pairs are drawn from the MUSE bilingual dictionary; duplicates and words absent from either vocabulary are removed. The cost is cosine distance,
\begin{equation}
C_{ij}=1-\langle x_i,y_j\rangle,
\label{eq:word-cost}
\end{equation}
rescaled to $[0,1]^{n\times n}$, with uniform marginals. Given the final coupling, the predicted French word for English word $i$ is $\arg\max_j P_{ij}$. We report top-1 and top-5 accuracy.

The vocabularies yield $232$ dictionary pairs. All seven solvers start at $v_0=0$ and use $\mathcal R(P)<10^{-6}$, a limit of $20{,}000$ iterations, and a five-minute time limit. We test $\varepsilon\in\{10^{-1},10^{-2},10^{-3},10^{-4}\}$. The converged top-1 accuracy is $82.8\%$, $81.9\%$, and $82.8\%$ at the three smallest values, respectively; reducing regularization does not consistently improve translation accuracy.

\begin{table}[tbp]
\centering
\caption{English--French alignment, $n=500$: oracle calls and wall time (seconds) to reach $\mathcal R(P)<10^{-6}$. Parentheses give Acc-Sinkhorn restarts; the bottom row gives converged-plan accuracy. 
}
\label{tab:word-alignment}
\renewcommand{\arraystretch}{1.125}
\setlength{\tabcolsep}{2pt}\resizebox{0.7\textwidth}{!}{
\begin{tabular}{lrrrrrrrr}
\toprule
& \multicolumn{2}{c}{$\varepsilon=10^{-1}$}
& \multicolumn{2}{c}{$\varepsilon=10^{-2}$}
& \multicolumn{2}{c}{$\varepsilon=10^{-3}$}
& \multicolumn{2}{c}{$\varepsilon=10^{-4}$}\\
\cmidrule(lr){2-3}\cmidrule(lr){4-5}\cmidrule(lr){6-7}\cmidrule(lr){8-9}
Method & Calls & Time & Calls & Time & Calls & Time & Calls & Time\\
\midrule
Sinkhorn & \textbf{3} & 0.01 & 270 & 0.51 & 939 & 1.93 & 7{,}277 & 11.89\\
Acc-Sinkhorn & 6 & 0.01 & \textbf{52} & \textbf{0.11} & \textbf{110} & \textbf{0.22} & \textbf{325} (2) & \textbf{0.57}\\
Overrelaxed Sinkhorn & \textbf{3} & 0.01 & 156 & 0.31 & 395 & 0.88 & 1{,}639 & 2.84\\
APDAGD & 68 & 0.13 & 354 & 0.68 & 1{,}136 & 2.36 & 4{,}398 & 7.29\\
APDAMD & 320 & 0.70 & 522 & 1.13 & 1{,}856 & 3.72 & 7{,}390 & 12.47\\
Scaling0 & 46 & 0.09 & 720 & 1.42 & 3{,}278 & 6.73 & 36{,}068 & 59.14\\
Annealed Sinkhorn & \textbf{3} & 0.01 & 305 & 0.58 & 6{,}241 & 13.21 & NR & --\\
\midrule
Converged top-1/5 (\%) & \multicolumn{2}{c}{78.0/94.8} & \multicolumn{2}{c}{82.8/95.3} & \multicolumn{2}{c}{81.9/95.7} & \multicolumn{2}{c}{82.8/94.8}\\
\bottomrule
\end{tabular}
}
\end{table}

\begin{figure}[tbp]
\centering
\includegraphics[width=0.8\linewidth]{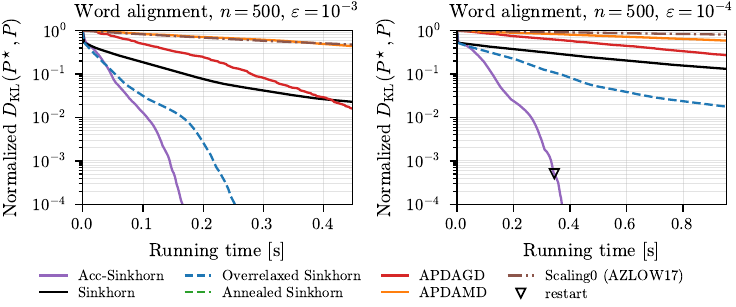}
\caption{Word alignment, $n=500$. Left: $\varepsilon=10^{-3}$; right: $10^{-4}$. Triangles mark Acc-Sinkhorn restarts; the second restart at $10^{-4}$ is below the plotted range. Call counts: Table~\ref{tab:word-alignment}.}
\label{fig:word-kl}
\end{figure}

\paragraph{Cost}
At $\varepsilon=10^{-1}$, Sinkhorn, overrelaxed Sinkhorn, and annealed Sinkhorn need three calls, and Acc-Sinkhorn needs six. Acc-Sinkhorn is fastest at the three smaller regularization values. At $\varepsilon=10^{-4}$, it uses about $22$ times fewer calls than Sinkhorn and $5$ times fewer than overrelaxed Sinkhorn, with runtime speedups of about $21$ and $5$. Annealed Sinkhorn does not reach the tolerance within the budget.

\section{Conclusion}
\label{sec:conclusion}

Acc-Sinkhorn accelerates the row-reduced Sinkhorn map while retaining its stable log-domain evaluation. Each accepted momentum step uses one Sinkhorn map, and the homotopy preserves the auxiliary state across stages. The numerical results show the largest gains at small regularization, where ordinary Sinkhorn converges slowly. Given Sinkhorn's widespread use in data science for comparing and aligning distributions, Acc-Sinkhorn offers a practical way to reduce the computational cost of these tasks.

The descent test adds no matrix pass when it succeeds and guarantees global linear convergence. In the reported tests, restarts are infrequent and prevent the observed instability at small regularization. For a fixed parameter, the momentum iteration has a local accelerated linear rate. Appendix~\ref{app:local-convergence} gives a global $\mathcal O(k^{-2})$ estimate under bounded restarts per stage. A uniform accelerated rate for the full homotopy method remains open.

\appendix
\section{A global homotopy estimate}
\label{app:local-convergence}
This section proves Theorem~\ref{thm:glob-rate} for Algorithm~\ref{alg:homotopy-acc}. Stage $s$ uses $\gamma_s=2^{-s}$ and $\alpha_s=\sqrt{2\gamma_s}$, and it has $M_s=\lceil\log4/\log(1+\alpha_s)\rceil$ accepted momentum steps.

\begin{theorem}[Global $\mathcal O(k^{-2})$ rate]
\label{appendix:glob-rate}
Suppose each stage of Algorithm~\ref{alg:homotopy-acc} has at most $F$ restarts. Then there is a constant $C>0$, depending on $F$, $x_0$, and the problem, such that
\begin{equation}
f\bigl(\Smap(x_k)\bigr)-f(x^\star)\leq\frac{C}{(1+k)^2},
\qquad k\geq0,
\tag{\ref{eq:glob-rate}}
\end{equation}
where $k$ counts both momentum steps and restarts.
\end{theorem}

The proof has two parts. Global linear convergence bounds the accumulated metric change. The stage estimates then give $\mathcal H_k=\mathcal O(\gamma_s)$, and the iteration count converts this bound to $\mathcal O(k^{-2})$. The constant can depend on the global curvature.

\subsection{Sinkhorn metric and energy}
For a positive diagonal $\mathcal D$, define on $\mathsf V$ the quotient norm and inner product
\begin{equation}
\|z\|_{\mathcal D}^2:=\min_{t\in\R}(z+t\one)^\top\mathcal D(z+t\one),
\qquad
\langle z,w\rangle_{\mathcal D}:=(\pi_{\mathcal D}z)^\top\mathcal D(\pi_{\mathcal D}w),
\label{eq:rs-quotient-metric}
\end{equation}
where $\pi_{\mathcal D}z:=z-\frac{\one^\top\mathcal Dz}{\one^\top\mathcal D\one}\one$ attains the minimum. Integrating \eqref{eq:dual-mirror-hessian} from $0$ to $\nabla f(x)$ gives
\begin{equation}
\nabla\phi^*(\nabla f(x))=\mathcal D(x)^{-1}\nabla f(x),
\qquad
\mathcal D(x):=\diag\bigl(\lm(c(x),b)\bigr),
\label{eq:precond-form}
\end{equation}
where $\lm(t,b):=(t-b)/(\log t-\log b)$, with $\lm(b,b)=b$, is the logarithmic mean. Recall $g(x):=x-\Smap(x)=\pi_{\mathsf V}\nabla\phi^*(\nabla f(x))$. 

Since $\one^\top\nabla f(x)=0$, we have $\pi_{\mathcal D(x)}g(x)=\mathcal D(x)^{-1}\nabla f(x)$. For $z\in\mathsf V$, this gives the first identity below. The second follows from the symmetric Bregman identity and $\nabla\phi^*(0)=0$:
\begin{equation}
\begin{aligned}
\langle z,g(x)\rangle_{\mathcal D(x)}&=\langle z,\nabla f(x)\rangle,
\\
\|g(x)\|_{\mathcal D(x)}^2&=\langle\nabla f(x),g(x)\rangle=D_{\phi^*}(\nabla f(x),0)+\KL(b\|c(x)).
\label{eq:rs-metric-identities}
\end{aligned}
\end{equation}
Thus $g(x)$ is the gradient of $f$ in the metric $\mathcal D(x)$. Write $\mathcal D_k:=\mathcal D(x_k)$. In a stage with parameter $\gamma$, define the energy
\begin{equation}
\mathcal H_k:=f(x_k)-f(x^\star)-\KL(b\|c(x_k))+\frac\gamma2\|y_k-x^\star\|_{\mathcal D_k}^2.
\label{eq:rs-energy}
\end{equation}
By Lemma~\ref{lem:sinkhorn-descent}, $f(x_k)-f(x^\star)-\KL(b\|c(x_k))\geq f(\Smap(x_k))-f(x^\star)\geq0$. Hence $\mathcal H_k\geq f(\Smap(x_k))-f(x^\star)$ and $\mathcal H_k\geq\frac\gamma2\|y_k-x^\star\|_{\mathcal D_k}^2$. We write $\mathcal H_k^\gamma$ when the parameter matters. At a fixed state, halving $\gamma$ decreases the energy:
\begin{equation}
\mathcal H_k^{\gamma/2}\leq\mathcal H_k^\gamma.
\label{eq:rs-energy-halving}
\end{equation}

\subsection{Bounds from global convergence}
Theorem~\ref{thm:rs-global} keeps every $x_k$ in the compact, convex sublevel set $\mathcal K_0=\{x\in\mathsf V:f(x)\leq f(x_0)\}$. Since $\mathcal D$ is continuous and positive there,
\begin{equation*}
R^2:=\max_{x\in\mathcal K_0}\|x-x^\star\|_{\mathcal D(x)}^2<\infty.
\end{equation*}
The logarithmic mean is smooth for positive arguments, including at equal arguments. Thus each diagonal entry $\mathcal D_{jj}(x)$ is positive and continuously differentiable on $\mathcal K_0$. Set
\begin{equation*}
L:=\frac{\displaystyle\max_{x\in\mathcal K_0,\,j}\|\nabla\mathcal D_{jj}(x)\|_2}{\displaystyle\min_{x\in\mathcal K_0,\,j}\mathcal D_{jj}(x)}<\infty,
\end{equation*}
where gradients are taken on $\mathsf V$. The mean value theorem along the segment from $x$ to $x'$ gives $\mathcal D_{jj}(x')-\mathcal D_{jj}(x)\leq L\|x'-x\|_2\mathcal D_{jj}(x)$ for each $j$. Hence
\begin{equation}
\mathcal D(x')-\mathcal D(x)\preceq L\|x'-x\|_2\mathcal D(x),
\qquad x,x'\in\mathcal K_0.
\label{eq:glob-metric}
\end{equation}
Theorem~\ref{thm:rs-global} also gives $\|x_k-x^\star\|_2\leq C(1-\sigma\mu)^{k/2}$. The triangle inequality therefore yields
\begin{equation*}
\sum_{k\geq0}\|x_{k+1}-x_k\|_2\leq2\sum_{k\geq0}\|x_k-x^\star\|_2<\infty.
\end{equation*}
This finite path length bounds the accumulated metric change. Below, $C$ denotes a finite constant independent of the iteration and stage indices; its value may change between estimates.

\subsection{Energy change within a stage}
We fix $\gamma$ and omit it from $\mathcal H_k$. A restart can increase the energy by at most $\gamma R^2/2$. A momentum step contracts it up to a metric factor and an additive term.

\begin{lemma}[Restart step]
\label{lem:rs-restart}

For a restart,
\begin{equation}
\mathcal H_{k+1}
\leq
\mathcal H_k-\frac\gamma2\|y_k-x^\star\|_{\mathcal D_k}^2
+\frac\gamma2\|x_{k+1}-x^\star\|_{\mathcal D_{k+1}}^2-\|g(x_{k+1})\|_{\mathcal D_{k+1}}^2.
\label{eq:loc-restart-exact}
\end{equation}
\end{lemma}

\begin{proof}
Here $x_{k+1}=\Smap(x_k)$ and $y_{k+1}=x_{k+1}$. Lemma~\ref{lem:sinkhorn-descent} gives
\begin{equation*}
f(x_{k+1})\leq f(x_k)-\KL(b\|c(x_k))-D_{\phi^*}(\nabla f(x_{k+1}),0).
\end{equation*}
The identities \eqref{eq:rs-metric-identities} give $D_{\phi^*}(\nabla f(x_{k+1}),0)+\KL(b\|c(x_{k+1}))=\|g(x_{k+1})\|_{\mathcal D_{k+1}}^2$. Subtracting $f(x^\star)+\KL(b\|c(x_{k+1}))$ from both sides,
\begin{equation*}
f(x_{k+1})-f(x^\star)-\KL(b\|c(x_{k+1}))\leq f(x_k)-f(x^\star)-\KL(b\|c(x_k))-\|g(x_{k+1})\|_{\mathcal D_{k+1}}^2.
\end{equation*}
The left side is $\mathcal H_{k+1}-\frac\gamma2\|x_{k+1}-x^\star\|_{\mathcal D_{k+1}}^2$, and the first three terms on the right are $\mathcal H_k-\frac\gamma2\|y_k-x^\star\|_{\mathcal D_k}^2$. This is \eqref{eq:loc-restart-exact}.
\end{proof}

The positive term in \eqref{eq:loc-restart-exact} comes from resetting $y_{k+1}=x_{k+1}$. The squared direction norm combines the forward and reverse KL divergences at $x_{k+1}$.

\begin{lemma}[Lyapunov estimate]
\label{lem:rs-momentum}
For a momentum step with $\alpha^2=2\gamma$,
\begin{equation}
\begin{aligned}
(1+\alpha)\mathcal H_{k+1}
\leq{}&
\mathcal H_k
+\frac\gamma2\Bigl(\|y_k-x^\star\|_{\mathcal D_{k+1}}^2-\|y_k-x^\star\|_{\mathcal D_k}^2\Bigr)\\
&+\alpha\Bigl(\frac{\gamma}2\|x_{k+1}-x^\star\|_{\mathcal D_{k+1}}^2-D_f(x^\star,x_{k+1})-\KL(b\|c(x_{k+1}))\Bigr).
\end{aligned}
\label{eq:rs-one-step}
\end{equation}
\end{lemma}

\begin{proof}
Freeze the metric at $\mathcal D_{k+1}$ and define
\begin{equation*}
\begin{aligned}
\mathcal E_{k, k+1} &:=f(x_k)-f(x^\star)+\frac\gamma2\|y_k-x^\star\|_{\mathcal D_{k+1}}^2, \\
\mathcal E_{k+1} &:=f(x_{k+1})-f(x^\star)+\frac\gamma2\|y_{k+1}-x^\star\|_{\mathcal D_{k+1}}^2.
\end{aligned}
\end{equation*}
These auxiliary energies satisfy $\mathcal E_{k+1}=\mathcal H_{k+1}+\KL(b\|c(x_{k+1}))$ and
\begin{equation*}
\mathcal E_{k,k+1}-\mathcal H_k-\KL(b\|c(x_k))
=\frac\gamma2\Bigl(\|y_k-x^\star\|_{\mathcal D_{k+1}}^2-\|y_k-x^\star\|_{\mathcal D_k}^2\Bigr).
\end{equation*}
By \eqref{eq:acc-sinkhorn-scheme} with $\alpha\beta=1$ and $\alpha^2=2\gamma$,
\begin{equation*}
x_{k+1}-x_k=\alpha(y_k-x_{k+1})-g(x_k),
\qquad
\gamma(y_{k+1}-y_k)=\alpha\gamma(x_{k+1}-y_{k+1})-\alpha g(x_{k+1}).
\end{equation*}
The Bregman identity for $f$ and the first equation give
\begin{equation}
f(x_{k+1})-f(x_k)
=
\alpha\langle\nabla f(x_{k+1}),y_k-x_{k+1}\rangle
-\langle\nabla f(x_{k+1}),g(x_k)\rangle
-D_f(x_k,x_{k+1}).
\label{eq:app-function-change}
\end{equation}
The second equation and \eqref{eq:rs-metric-identities} give
\begin{equation}
\begin{aligned}
&\frac\gamma2\Bigl(\|y_{k+1}-x^\star\|_{\mathcal D_{k+1}}^2-\|y_k-x^\star\|_{\mathcal D_{k+1}}^2\Bigr)\\
={}&
\alpha\gamma\langle y_{k+1}-x^\star,x_{k+1}-y_{k+1}\rangle_{\mathcal D_{k+1}}
-\alpha\langle\nabla f(x_{k+1}),y_{k+1}-x^\star\rangle -\frac\gamma2\|y_{k+1}-y_k\|_{\mathcal D_{k+1}}^2.
\end{aligned}
\label{eq:app-quadratic-change}
\end{equation}
The left sides of \eqref{eq:app-function-change} and \eqref{eq:app-quadratic-change} add up to $\mathcal E_{k+1}-\mathcal E_{k, k+1}$. The two gradient terms combine through
\begin{equation*}
(y_k-x_{k+1})-(y_{k+1}-x^\star)=(y_k-y_{k+1})+(x^\star-x_{k+1}),
\end{equation*}
and the Bregman identity for $f$ at $x_{k+1}$ gives
\begin{equation*}
\alpha\langle\nabla f(x_{k+1}),x^\star-x_{k+1}\rangle=-\alpha\bigl(f(x_{k+1})-f(x^\star)\bigr)-\alpha D_f(x^\star,x_{k+1}).
\end{equation*}
Adding $\alpha\mathcal E_{k+1}$ to both sides cancels $-\alpha(f(x_{k+1})-f(x^\star))$ and leaves $\frac{\alpha\gamma}2\|y_{k+1}-x^\star\|_{\mathcal D_{k+1}}^2$. With
\begin{equation*}
\begin{aligned}
T_1&:=\alpha\langle\nabla f(x_{k+1}),y_k-y_{k+1}\rangle-\frac\gamma2\|y_{k+1}-y_k\|_{\mathcal D_{k+1}}^2,\\
T_2&:=\langle y_{k+1}-x^\star,x_{k+1}-y_{k+1}\rangle_{\mathcal D_{k+1}}+\frac12\|y_{k+1}-x^\star\|_{\mathcal D_{k+1}}^2,
\end{aligned}
\end{equation*}
we arrive at the exact relation
\begin{equation}
(1+\alpha)\mathcal E_{k+1}-\mathcal E_{k, k+1}
=
-\langle\nabla f(x_{k+1}),g(x_k)\rangle-D_f(x_k,x_{k+1})
-\alpha D_f(x^\star,x_{k+1})
+T_1+\alpha\gamma T_2.
\label{eq:app-exact-split}
\end{equation}
For $T_1$, \eqref{eq:rs-metric-identities} and Young's inequality give
\begin{equation*}
T_1
\leq\frac{\alpha^2}{2\gamma}\|g(x_{k+1})\|_{\mathcal D_{k+1}}^2=\|g(x_{k+1})\|_{\mathcal D_{k+1}}^2.
\end{equation*}
For $T_2$, expanding the square along $x_{k+1}-x^\star=(y_{k+1}-x^\star)+(x_{k+1}-y_{k+1})$ gives
\begin{equation*}
T_2=\frac12\|x_{k+1}-x^\star\|_{\mathcal D_{k+1}}^2-\frac12\|x_{k+1}-y_{k+1}\|_{\mathcal D_{k+1}}^2\leq\frac12\|x_{k+1}-x^\star\|_{\mathcal D_{k+1}}^2.
\end{equation*}
For the cross term, since $\one^\top\nabla f(x_{k+1})=0$, the three-point identity for $\phi^*$ gives
\begin{equation*}
\begin{aligned}
\langle\nabla f(x_{k+1}),g(x_k)\rangle
&=\langle\nabla f(x_{k+1}),\nabla\phi^*(\nabla f(x_k))\rangle\\
&=D_{\phi^*}(\nabla f(x_{k+1}),0)+D_{\phi^*}(0,\nabla f(x_k))\\
&\quad-D_{\phi^*}(\nabla f(x_{k+1}),\nabla f(x_k)).
\end{aligned}
\end{equation*}
Corollary~\ref{cor:dual-relative-smoothness-bregman} and Bregman duality give 
\begin{equation*}
D_{\phi^*}(\nabla f(x_{k+1}),\nabla f(x_k))\leq D_{f^*}(\nabla f(x_{k+1}),\nabla f(x_k))=D_f(x_k,x_{k+1}).
\end{equation*} 
Using \eqref{eq:rs-metric-identities}, the bound for $T_1$ is $D_{\phi^*}(\nabla f(x_{k+1}),0)+\KL(b\|c(x_{k+1}))$. The first term cancels its negative counterpart in the cross term. Substituting these bounds into \eqref{eq:app-exact-split} gives
\begin{equation*}
\begin{aligned}
(1+\alpha)\mathcal E_{k+1}
&\leq
\mathcal E_{k, k+1}-\KL(b\|c(x_k))+\KL(b\|c(x_{k+1}))\\
&\quad +\alpha\Bigl(\frac\gamma2\|x_{k+1}-x^\star\|_{\mathcal D_{k+1}}^2-D_f(x^\star,x_{k+1})\Bigr).
\end{aligned}
\end{equation*}
Subtracting $(1+\alpha)\KL(b\|c(x_{k+1}))$ from both sides and inserting the expressions for $\mathcal E_{k+1}$ and $\mathcal E_{k, k+1}$ in terms of $\mathcal H_{k+1}$ and $\mathcal H_k$ gives \eqref{eq:rs-one-step}.
\end{proof}

The last two terms in \eqref{eq:rs-one-step} account for metric changes and curvature. The metric bound \eqref{eq:glob-metric} controls the first. It also bounds the quotient norms: evaluate the new quadratic form at the constant shift that minimizes the old one. Using $\frac\gamma2\|y_k-x^\star\|_{\mathcal D_k}^2\leq\mathcal H_k$, we obtain
\begin{equation}
\begin{aligned}
\frac\gamma2\Bigl(\|y_k-x^\star\|_{\mathcal D_{k+1}}^2-\|y_k-x^\star\|_{\mathcal D_k}^2\Bigr)
&\leq
L\|x_{k+1}-x_k\|_2\,\frac\gamma2\|y_k-x^\star\|_{\mathcal D_k}^2\\
&\leq
L\|x_{k+1}-x_k\|_2\,\mathcal H_k.
\end{aligned}
\label{eq:glob-metric-term}
\end{equation}

\begin{lemma}[One-step perturbed estimate]
\label{lem:glob-one-step}
A momentum step satisfies
\begin{equation}
(1+\alpha)\mathcal H_{k+1}\leq(1+L\|x_{k+1}-x_k\|_2)\mathcal H_k+\frac{\alpha\gamma}2R^2,
\label{eq:glob-momentum}
\end{equation}
and a restart satisfies
\begin{equation}
\mathcal H_{k+1}\leq\mathcal H_k+\frac\gamma2R^2.
\label{eq:glob-restart}
\end{equation}
\end{lemma}

\begin{proof}
Apply \eqref{eq:glob-metric-term} to the metric term in \eqref{eq:rs-one-step}. Drop the nonpositive terms $-D_f(x^\star,x_{k+1})$ and $-\KL(b\|c(x_{k+1}))$, and use $\|x_{k+1}-x^\star\|_{\mathcal D_{k+1}}^2\leq R^2$. This gives \eqref{eq:glob-momentum}. In \eqref{eq:loc-restart-exact}, drop the two nonpositive terms and use $\|x_{k+1}-x^\star\|_{\mathcal D_{k+1}}^2\leq R^2$. This gives \eqref{eq:glob-restart}.
\end{proof}

Recall that $N_s$ is the first index of stage $s$. Let $F_s$ be its number of restarts, so $F_s\leq M_s+1$ by Lemma~\ref{lem:rs-no-consecutive}. To sum the one-step bounds, set
\begin{equation*}
P_s:=\prod_{j=N_s}^{N_{s+1}-1}\bigl(1+L\|x_{j+1}-x_j\|_2\bigr).
\end{equation*}
The path-length bound gives
\begin{equation*}
\prod_{s\geq0}P_s\leq\exp\Bigl(L\sum_{j\geq0}\|x_{j+1}-x_j\|_2\Bigr)\leq\exp\Bigl(2L\sum_{j\geq0}\|x_j-x^\star\|_2\Bigr)<\infty.
\end{equation*}

\begin{lemma}[Fixed-stage estimate]
\label{lem:glob-fixed-stage}
For $N_s\leq k\leq N_{s+1}$, let $j$ count the accepted momentum steps from $N_s$ to $k-1$. Then
\begin{equation}
\mathcal H_k^{\gamma_s}
\leq
P_s\Bigl((1+\alpha_s)^{-j}\mathcal H_{N_s}^{\gamma_s}+(1+F_s)\frac{\gamma_s}2R^2\Bigr).
\label{eq:glob-fixed-stage}
\end{equation}
\end{lemma}

\begin{proof}
Iterate Lemma~\ref{lem:glob-one-step}, bounding every product of metric factors by $P_s$. The initial energy is multiplied by $(1+\alpha_s)^{-j}$. Each restart adds at most $\gamma_sR^2/2$, and the momentum contributions sum to at most
\begin{equation*}
\frac{\gamma_sR^2}2\sum_{m=1}^j\frac{\alpha_s}{(1+\alpha_s)^m}\leq\frac{\gamma_sR^2}2,
\end{equation*}
where the momentum steps are counted backward from $k$. Multiplying these bounds by $P_s$ proves the claim.
\end{proof}

\subsection{From stages to iterations}
We first relate the stage parameter to the iteration count, then show that the energy remains of order $\gamma_s$.

\begin{lemma}[Counting]
\label{lem:counting}
For every $s\geq0$,
\begin{equation}
N_{s+1}\leq\frac{27}{\sqrt{\gamma_s}}.
\label{eq:counting}
\end{equation}
\end{lemma}

\begin{proof}
Stage $j$ has $M_j$ momentum steps and at most $M_j+1$ restarts. Since $M_j\leq1+\log4/\log(1+\alpha_j)$, $\log(1+\alpha)\geq\alpha/(1+\alpha)$, and $\alpha_j\leq\sqrt2$,
\begin{equation*}
N_{j+1}-N_j\leq2M_j+1\leq3+2\log4\,\frac{1+\alpha_j}{\alpha_j}\leq\frac{11}{\alpha_j}.
\end{equation*}
Since $\alpha_j=2^{(s-j)/2}\alpha_s$, summing this geometric series over $0\leq j\leq s$ gives
\[
N_{s+1}\leq\frac{11}{\alpha_s}\sum_{i\geq0}2^{-i/2}=\frac{11(2+\sqrt2)}{\sqrt{2\gamma_s}}\leq\frac{27}{\sqrt{\gamma_s}}.
\]
\end{proof}

For $N_s\leq k<N_{s+1}$, integer indices give $1+k\leq N_{s+1}$. Thus \eqref{eq:counting} implies
\begin{equation*}
\gamma_s\leq\frac{27^2}{N_{s+1}^2}\leq\frac{27^2}{(1+k)^2}.
\end{equation*}
It remains to bound $\mathcal H_k^{\gamma_s}$ by a constant times $\gamma_s$.

\begin{lemma}[Homotopy stage estimate]
\label{lem:glob-homotopy-stage}
Assume that $F_s\leq F$ for all $s$. Then, for every $s\geq0$ and $N_s\leq k<N_{s+1}$,
\begin{equation}
\mathcal H_k^{\gamma_s}\leq C\gamma_s.
\label{eq:glob-homotopy-stage}
\end{equation}
\end{lemma}

\begin{proof}
At the end of a stage, $(1+\alpha_s)^{-M_s}\leq1/4$. Lemma~\ref{lem:glob-fixed-stage} and the energy decrease when $\gamma$ is halved give
\begin{equation*}
\mathcal H_{N_{s+1}}^{\gamma_{s+1}}\leq P_s\Bigl(\frac14\mathcal H_{N_s}^{\gamma_s}+\frac{(1+F)\gamma_sR^2}2\Bigr).
\end{equation*}
Divide by $\gamma_{s+1}\prod_{t=0}^{s}P_t$ and use $\gamma_{s+1}=\gamma_s/2$. Since every $P_t\geq1$, this gives
\begin{equation*}
\frac{\mathcal H_{N_{s+1}}^{\gamma_{s+1}}}{\gamma_{s+1}\prod_{t=0}^{s}P_t}
\leq\frac12\frac{\mathcal H_{N_s}^{\gamma_s}}{\gamma_s\prod_{t=0}^{s-1}P_t}+(1+F)R^2.
\end{equation*}
Summing this recurrence yields
\begin{equation}
\frac{\mathcal H_{N_s}^{\gamma_s}}{\gamma_s\prod_{t=0}^{s-1}P_t}\leq\frac{\mathcal H_0^{\gamma_0}}{\gamma_0}+2(1+F)R^2.
\label{eq:glob-stage-start}
\end{equation}
The empty product at $s=0$ is one. Since $\prod_tP_t<\infty$, the stage-start energy is bounded by $C\gamma_s$. Lemma~\ref{lem:glob-fixed-stage} extends this bound to every iterate in the stage.
\end{proof}

\begin{proof}[Proof of Theorem~\ref{thm:glob-rate}]
For $N_s\leq k<N_{s+1}$, the energy bound and Lemma~\ref{lem:counting} give
\[
f\bigl(\Smap(x_k)\bigr)-f(x^\star)\leq\mathcal H_k^{\gamma_s}\leq C\gamma_s\leq\frac{27^2C}{(1+k)^2}.
\]
\end{proof}

The bound on $F_s$ controls the additive restart terms in this energy argument. The global linear estimate in Theorem~\ref{thm:rs-global} holds without it. The reported experiments have few restarts per stage, but a uniform bound on $F_s$ independent of $s$ remains unproved. 

\section{Experimental details}
\label{sec:experimental-details}

\paragraph{Baseline parameters}
No baseline was tuned on the test instances. Each is run at the setting its own paper prescribes.
\begin{itemize}[leftmargin=2em]
\item {\it Sinkhorn} has no parameters.
\item {\it Overrelaxed Sinkhorn} uses the adaptive rule $\Theta(w)=\min\{\max\{1,\Theta^\star(w)-\delta\},\theta_0\}$ of~\cite{thibault2021overrelaxed} with safeguard $\delta=0.05$, and $\theta_0=2/(1+\sqrt\eta)$ from their Proposition~2. Following their Section~5, we use the realistic {\it estimated} strategy for $\eta$: plain Sinkhorn is run on a held-out instance with the same $(n,m,\varepsilon)$ but a different seed, and the local rate is read off a late window of that run rather than from a short warm-up, which would bias $\theta_0$ low. The test instance is never used for calibration, and the calibration cost is not charged to the reported run.
\item {\it Annealed Sinkhorn} uses the schedule $\beta_t=\min\{\beta_0(1+t)^\kappa,1\}$ of~\cite{chizat2024annealed} with $\kappa=1/2$ and $\beta_0=\min\{10/\|C\|_{\mathrm{osc}},1\}$, where $\|C\|_{\mathrm{osc}}=\max_{ij}C_{ij}-\min_{ij}C_{ij}$. The cap at $1$ is needed because the target here is the entropic problem at fixed $\varepsilon$, that is, $\beta=1$, not the unregularized limit.
\item {\it APDAGD} and {\it APDAMD} start from $L_0=1$ and adapt through their own backtracking tests; APDAMD uses $\delta=n+m$, the dual dimension. Both act on the full two-marginal dual, which is what makes their gradient Lipschitz.
\item {\it Scaling0} uses the a priori diameter bound of~\cite{allenzhu2017scaling}, Lemma~3.2, namely $N=\max_{ij}\bar C_{ij}+\log(nm)$, a warm start of $\lceil\log_2N\rceil+1$ rounds of $\lceil N\rceil$ iterations, and equal budgets for the warm start, the main linear-coupling call, and the gradient phase.
\end{itemize}
Acc-Sinkhorn's own setting is the fixed $\gamma_0=1$ of Algorithm~\ref{alg:homotopy-acc}, from which $M_\ell$ is derived, likewise untuned per instance.

\paragraph{Protocol}
In every study, all solvers share the instance, the starting potential $v_0$, and the log-domain primitives of Section~\ref{sec:stable-evaluation}. The synthetic problems stop at $\mathcal R(P)=\KL(a\|P\one_m)+\KL(b\|P^\top\one_n)<10^{-6}$ or $200{,}000$ iterations and are run on seeds $11$, $23$, and $47$. Timers cover algorithmic work only: the reference solves are excluded, as is the one pass per iteration that overrelaxed Sinkhorn, annealed Sinkhorn, and APDAMD spend converting their iterate into the reported quantity, which their own iterations do not need.

\paragraph{Reproducibility}
Each study writes its curves and its scalar results to disk as it runs, after every regularization level, so that tables and figures can be rebuilt without solving anything; each driver has a replot mode that does exactly that. The reference solves, which are the most expensive single step in the two applications, are cached by their dual potential and reused across runs.

\bibliographystyle{siam}
\bibliography{references,references_acc_additions}

\begin{thebibliography}{10}

\bibitem{allenzhu2017scaling}
{\sc Z.~Allen-Zhu, Y.~Li, R.~Oliveira, and A.~Wigderson}, {\em Much faster
  algorithms for matrix scaling}, in 2017 IEEE 58th Annual Symposium on
  Foundations of Computer Science (FOCS), IEEE, 2017, pp.~890--901.

\bibitem{alvarez2020unsupervised}
{\sc D.~Alvarez-Melis, Y.~Mroueh, and T.~Jaakkola}, {\em Unsupervised hierarchy
  matching with optimal transport over hyperbolic spaces}, in Proceedings of
  the 23rd International Conference on Artificial Intelligence and Statistics,
  2020, pp.~1606--1617.

\bibitem{AttouchChbaniFadiliRiahi2022First-order}
{\sc H.~Attouch, Z.~Chbani, J.~Fadili, and H.~Riahi}, {\em First-order
  optimization algorithms via inertial systems with {Hessian} driven damping},
  Mathematical Programming, 193 (2022), pp.~113--155.

\bibitem{aubin2022mirror}
{\sc P.-C. Aubin-Frankowski, A.~Korba, and F.~L{\'e}ger}, {\em Mirror descent
  with relative smoothness in measure spaces, with application to {Sinkhorn}
  and {EM}}, in Advances in Neural Information Processing Systems, vol.~35,
  2022, pp.~17263--17275.

\bibitem{benamou2015iterative}
{\sc J.-D. Benamou, G.~Carlier, M.~Cuturi, L.~Nenna, and G.~Peyr{\'e}}, {\em
  Iterative {Bregman} projections for regularized transportation problems},
  SIAM Journal on Scientific Computing, 37 (2015), pp.~A1111--A1138.

\bibitem{bojanowski2017enriching}
{\sc P.~Bojanowski, E.~Grave, A.~Joulin, and T.~Mikolov}, {\em Enriching word
  vectors with subword information}, Transactions of the Association for
  Computational Linguistics, 5 (2017), pp.~135--146.

\bibitem{carlier2022linear}
{\sc G.~Carlier}, {\em On the linear convergence of the multimarginal
  {Sinkhorn} algorithm}, SIAM Journal on Optimization, 32 (2022), pp.~786--794.

\bibitem{ChenHaoWei2025Accelerated}
{\sc L.~Chen, L.~Hao, and J.~Wei}, {\em Accelerated gradient methods through
  variable and operator splitting}, arXiv preprint arXiv:2505.04065,  (2025).

\bibitem{chen2019orderoptimizationmethodsbased}
{\sc L.~Chen and H.~Luo}, {\em First order optimization methods based on
  {Hessian}-driven {Nesterov} accelerated gradient flow}, arXiv preprint
  arXiv:1912.09276,  (2019).

\bibitem{chen2025hnag}
{\sc L.~Chen and Z.~Xu}, {\em {HNAG$^{++}$}: An accelerated gradient method
  with a refined asymptotic rate for strongly convex optimization}, arXiv
  preprint arXiv:2510.16680,  (2025).

\bibitem{chizat2024annealed}
{\sc L.~Chizat}, {\em Annealed {Sinkhorn} for optimal transport: Convergence,
  regularization path and debiasing}, arXiv preprint arXiv:2408.11620,  (2024).

\bibitem{chizat2026sharper}
{\sc L.~Chizat, A.~Delalande, and T.~Va{\v{s}}kevi{\v{c}}ius}, {\em Sharper
  exponential convergence rates for {Sinkhorn}'s algorithm in continuous
  settings}, Mathematical Programming, 215 (2026), pp.~809--858.

\bibitem{pmlr-v235-chopin24a}
{\sc N.~Chopin, F.~Crucinio, and A.~Korba}, {\em A connection between tempering
  and entropic mirror descent}, in Proceedings of the 41st International
  Conference on Machine Learning, vol.~235 of Proceedings of Machine Learning
  Research, 2024, pp.~8782--8800.

\bibitem{conneau2018word}
{\sc A.~Conneau, G.~Lample, M.~Ranzato, L.~Denoyer, and H.~J{\'e}gou}, {\em
  Word translation without parallel data}, in International Conference on
  Learning Representations, 2018.

\bibitem{courty2017joint}
{\sc N.~Courty, R.~Flamary, A.~Habrard, and A.~Rakotomamonjy}, {\em Joint
  distribution optimal transportation for domain adaptation}, in Advances in
  Neural Information Processing Systems, vol.~30, 2017.

\bibitem{cuturi2013sinkhorn}
{\sc M.~Cuturi}, {\em Sinkhorn distances: Lightspeed computation of optimal
  transport}, in Advances in Neural Information Processing Systems, vol.~26,
  2013.

\bibitem{dvurechensky2018computational}
{\sc P.~Dvurechensky, A.~Gasnikov, and A.~Kroshnin}, {\em Computational optimal
  transport: Complexity by accelerated gradient descent is better than by
  {Sinkhorn}'s algorithm}, in Proceedings of the 35th International Conference
  on Machine Learning, 2018, pp.~1367--1376.

\bibitem{ferradans2014regularized}
{\sc S.~Ferradans, N.~Papadakis, G.~Peyr{\'e}, and J.-F. Aujol}, {\em
  Regularized discrete optimal transport}, SIAM Journal on Imaging Sciences, 7
  (2014), pp.~1853--1882.

\bibitem{reinger2021pot}
{\sc R.~Flamary, N.~Courty, A.~Gramfort, M.~Z. Alaya, A.~Boisbunon, S.~Chambon,
  L.~Chapel, A.~Corenflos, K.~Fatras, N.~Fournier, L.~Gautheron, N.~T.~H.
  Gayraud, H.~Janati, A.~Rakotomamonjy, I.~Redko, A.~Rolet, A.~Schutz,
  V.~Seguy, D.~J. Sutherland, R.~Tavenard, A.~Tong, and T.~Vayer}, {\em {POT}:
  Python optimal transport}, Journal of Machine Learning Research, 22 (2021),
  pp.~1--8.

\bibitem{franklin1989scaling}
{\sc J.~Franklin and J.~Lorenz}, {\em On the scaling of multidimensional
  matrices}, Linear Algebra and its Applications, 114--115 (1989),
  pp.~717--735.

\bibitem{gao2026tight}
{\sc W.~Gao, Z.~Qu, Y.~Ye, and M.~Udell}, {\em Tight nonasymptotic local
  convergence of {Sinkhorn--Knopp}}, 2026.
\newblock arXiv:2608.11760.

\bibitem{genevay2018learning}
{\sc A.~Genevay, G.~Peyr{\'e}, and M.~Cuturi}, {\em Learning generative models
  with {Sinkhorn} divergences}, in Proceedings of the 21st International
  Conference on Artificial Intelligence and Statistics, 2018, pp.~1608--1617.

\bibitem{ghosal2025convergence}
{\sc P.~Ghosal and M.~Nutz}, {\em On the convergence rate of {Sinkhorn}'s
  algorithm}, Mathematics of Operations Research, 51 (2026), pp.~1080--1096.

\bibitem{guminov2021combination}
{\sc S.~Guminov, P.~Dvurechensky, N.~Tupitsa, and A.~Gasnikov}, {\em On a
  combination of alternating minimization and {Nesterov}'s momentum}, in
  Proceedings of the 38th International Conference on Machine Learning,
  vol.~139 of Proceedings of Machine Learning Research, 2021, pp.~3886--3898.

\bibitem{pmlr-v238-reza-karimi24a}
{\sc M.~R. Karimi, Y.-P. Hsieh, and A.~Krause}, {\em {Sinkhorn} flow as mirror
  flow: A continuous-time framework for generalizing the {Sinkhorn} algorithm},
  in Proceedings of the 27th International Conference on Artificial
  Intelligence and Statistics, vol.~238 of Proceedings of Machine Learning
  Research, 2024, pp.~4186--4194.

\bibitem{knight2008sinkhorn}
{\sc P.~A. Knight}, {\em The {Sinkhorn--Knopp} algorithm: Convergence and
  applications}, SIAM Journal on Matrix Analysis and Applications, 30 (2008),
  pp.~261--275.

\bibitem{leger2021gradient}
{\sc F.~L{\'e}ger}, {\em A gradient descent perspective on {Sinkhorn}}, Applied
  Mathematics and Optimization, 84 (2021), pp.~1843--1855.

\bibitem{lin2019efficient}
{\sc T.~Lin, N.~Ho, and M.~I. Jordan}, {\em On efficient optimal transport: An
  analysis of greedy and accelerated mirror descent algorithms}, in Proceedings
  of the 36th International Conference on Machine Learning, 2019,
  pp.~3982--3991.

\bibitem{lin2019acceleration}
\leavevmode\vrule height 2pt depth -1.6pt width 23pt, {\em On the efficiency of
  entropic regularized algorithms for optimal transport}, Journal of Machine
  Learning Research, 23 (2022), pp.~1--42.

\bibitem{maddison2021dual}
{\sc C.~J. Maddison, D.~Paulin, Y.~W. Teh, and A.~Doucet}, {\em Dual space
  preconditioning for gradient descent}, SIAM Journal on Optimization, 31
  (2021), pp.~991--1016.

\bibitem{mishchenko2019sinkhorn}
{\sc K.~Mishchenko}, {\em Sinkhorn algorithm as a special case of stochastic
  mirror descent}, arXiv preprint arXiv:1909.06918,  (2019).

\bibitem{odonoghue2015adaptive}
{\sc B.~O'Donoghue and E.~Cand\`es}, {\em Adaptive restart for accelerated
  gradient schemes}, Foundations of Computational Mathematics, 15 (2015),
  pp.~715--732.

\bibitem{peyre2019computational}
{\sc G.~Peyr{\'e} and M.~Cuturi}, {\em Computational Optimal Transport},
  vol.~11 of Foundations and Trends in Machine Learning, Now Publishers, 2019.

\bibitem{reinhard2001color}
{\sc E.~Reinhard, M.~Ashikhmin, B.~Gooch, and P.~Shirley}, {\em Color transfer
  between images}, IEEE Computer Graphics and Applications, 21 (2001),
  pp.~34--41.

\bibitem{schmitzer2019stabilized}
{\sc B.~Schmitzer}, {\em Stabilized sparse scaling algorithms for entropy
  regularized transport problems}, SIAM Journal on Scientific Computing, 41
  (2019), pp.~A1443--A1481.

\bibitem{sinkhorn1964}
{\sc R.~Sinkhorn}, {\em A relationship between arbitrary positive matrices and
  doubly stochastic matrices}, The Annals of Mathematical Statistics, 35
  (1964), pp.~876--879.

\bibitem{solomon2015convolutional}
{\sc J.~Solomon, F.~de~Goes, G.~Peyr{\'e}, M.~Cuturi, A.~Butscher, A.~Nguyen,
  T.~Du, and L.~Guibas}, {\em Convolutional {Wasserstein} distances: Efficient
  optimal transportation on geometric domains}, ACM Transactions on Graphics,
  34 (2015).

\bibitem{thibault2021overrelaxed}
{\sc A.~Thibault, L.~Chizat, C.~Dossal, and N.~Papadakis}, {\em Overrelaxed
  {Sinkhorn--Knopp} algorithm for regularized optimal transport}, Algorithms,
  14 (2021), p.~143.

\bibitem{xu2019learning}
{\sc L.~Xu, H.~Sun, and Y.~Liu}, {\em Learning with batch-wise optimal
  transport loss for {3D} shape recognition}, in Proceedings of the IEEE/CVF
  Conference on Computer Vision and Pattern Recognition, 2019, pp.~3333--3342.

\end{thebibliography}

\end{document}